\documentclass[11pt]{article}
\usepackage{amsmath,amsthm,amssymb}
\usepackage{enumerate}
\usepackage[hidelinks]{hyperref}
\usepackage{microtype}

\def\titlerunning#1{\gdef\titrun{#1}}
\makeatletter
\def\author#1{\gdef\autrun{\def\and{\unskip, }#1}\gdef\@author{#1}}
\def\address#1{{\def\and{\\\hspace*{18pt}}\renewcommand{\thefootnote}{}%
\footnote {#1}}%
\markboth{\autrun}{\titrun}}
\makeatother
\def\email#1{e-mail: #1}
\def\subjclass#1{{\renewcommand{\thefootnote}{}%
\footnote{\emph{Mathematics Subject Classification (2010):} #1}}}
\def\keywords#1{\par\medskip
\noindent\textbf{Keywords.} #1}

\numberwithin{equation}{section}
\newcommand{\defword}[1]{\textbf{#1}}

\renewcommand{\Re}{\operatorname{Re}}

\usepackage{color}
\newcommand{\len}{}
\newcommand{\nate}{}

\usepackage[numbers,sort]{natbib}
\usepackage{cancel}
\newtheorem{theorem}{Theorem}[section]

\newtheorem{assumption}[theorem]{Assumption}

\newtheorem{corollary}[theorem]{Corollary}

\newtheorem{lemma}[theorem]{Lemma}

\newtheorem{proposition}[theorem]{Proposition}

\newtheorem{question}[theorem]{Question}

\theoremstyle{definition}
\newtheorem{definition}[theorem]{Definition}
\newtheorem{example}[theorem]{Example}
\newtheorem{notation}[theorem]{Notation}

\theoremstyle{remark}
\newtheorem{remark}[theorem]{Remark}

\begin{document}

\titlerunning{Strong hypercontractivity}

\title{Strong hypercontractivity and logarithmic Sobolev inequalities
  on stratified complex Lie groups}

\author{Nathaniel Eldredge
\and 
Leonard Gross
\and
Laurent Saloff-Coste
}

\date{September 2, 2017}

\maketitle

\address{
N.~Eldredge: School of Mathematical Sciences, University of Northern Colorado, 501 20th
  St.~Box 122, Greeley, CO 80639 USA; \email{neldredge@unco.edu}
\and
L.~Gross: Department of Mathematics, Cornell University, 301 Malott
Hall, Ithaca, NY 14853 USA; \email{gross@math.cornell.edu}
\and
L.~Saloff-Coste: Department of Mathematics, Cornell University, 301 Malott
Hall, Ithaca, NY 14853 USA; \email{lsc@math.cornell.edu}
}


\begin{abstract}
  We show that for a hypoelliptic Dirichlet form operator $A$ on a
  stratified complex Lie group, if the logarithmic Sobolev inequality
  holds, then a holomorphic projection of $A$ is strongly
  hypercontractive in the sense of Janson.  This extends previous
  results of Gross to a setting in which the operator $A$ is not
  holomorphic. 
\end{abstract}

\tableofcontents

\section{Introduction}\label{sec:intro}

In \cite{driver-gross-hilbert-spaces-holomorphic, dgs1, dgs2, dgs3},
subsets of the current authors, together with Bruce K.~Driver, studied
properties of elliptic and hypoelliptic heat kernels on complex Lie
groups and homogeneous spaces, particularly the Taylor map for $L^2$
holomorphic functions.  Generally, it was shown that hypoelliptic heat
kernels and their sub-Laplacians often behave similarly to their
elliptic counterparts, such as the Gaussian heat kernel and standard
Laplacian on $\mathbb{C}^n$.  {\len In this paper we turn our
  attention to the phenomenon of strong hypercontractivity in the
  particular case of stratified complex Lie groups.}

To motivate this study, let us first consider Euclidean space
$\mathbb{R}^n$ equipped with standard Gaussian measure $\nu$.  Let
$Q(f,g)$ be the Dirichlet form with core $C^\infty_c(\mathbb{R}^n)$
defined by $Q(f,g) = \int_{\mathbb{R}^n} \nabla f \cdot \nabla
\bar{g}\,d{\len \nu}$, whose generator is the Ornstein--Uhlenbeck operator
$Af(x) = -\Delta f(x) + {\len x}\cdot \nabla f(x) $.  In
\cite{nelson-free-markov}, E.~Nelson discovered that
the semigroup $e^{-tA}$ enjoys the following property known as
\defword{hypercontractivity}:
\begin{theorem} \label{thmN}
  For $1 < q \le p < \infty$, let $t_N(p,q) = \frac{1}{2}
  \log\left(\frac{p-1}{q-1}\right)$.  {\len Then} for any $t \ge t_N$, $e^{-tA}$
  is a contraction from $L^q({\len \nu})$ to $L^p({\len \nu})$.  
\end{theorem}
So the semigroup $e^{-tA}$ improves {\len local} integrability of functions with
respect to ${\len \nu}$; as soon as $t$ exceeds ``Nelson's time'' $t_N(p,q)$,
$e^{-tA}$ maps $L^q$ into $L^p$.  Moreover, Nelson's time is sharp:
for $t < t_N(p,q)$, $e^{-tA}$ is unbounded from $L^q$ to $L^p$.
{\len  For a short history of this theorem, see the survey \cite{gross-lsi-survey-2006}.}

Now replace ${\len \nu}$ with any smooth measure  ${\len \mu} $ on $\mathbb{R}^n$
{\len and redefine $Q$ and $A$ accordingly}.
In \cite{gross75}, the second author introduced the
notion of a \defword{logarithmic Sobolev inequality}, which (in its
simplest version) is said to
be satisfied by $\mu$ if 
\begin{equation}\label{LSI-intro}
\int_{\mathbb{R}^n} |f|^2\log|f|\,d\mu
      \le Q(f)
	    + \|f\|^2_{L^2(\mu)} \log\|f\|_{L^2(\mu)}
\end{equation}
{ for all $f$ in the domain of $Q$.}

(Actually, in this paper, we shall study a more general version of
(\ref{LSI-intro}) in which the coefficient of $Q(f)$ is a constant $c$
other than 1, and in which a term of the form $\beta \|f\|_{L^2}^2$
can be added to the right side. See (\ref{LSI}).  The general version can
also be used in the theorems in this introduction, making appropriate
changes to the constants, but for simplicity we omit the details here.)

It was shown in \cite{gross75} that in this case the logarithmic
Sobolev inequality (\ref{LSI-intro}) is essentially equivalent to
hypercontractivity:
\begin{theorem} \label{thmequiv}
  A smooth measure $\mu$ on $\mathbb{R}^n$ satisfies the logarithmic
  Sobolev inequality (\ref{LSI-intro}) if and only if the
  corresponding semigroup $e^{-tA}$ is hypercontractive $($with Nelson's
  time $t_N$$)$. 
\end{theorem}

The early history of these two types of inequalities devolves
from two different backgrounds.
In 1959 A.~J.~Stam  \cite{stam59},   motivated by problems in information theory,
 proved an inequality, based on Lebesgue measure rather than Gauss measure,   
 easily transformable into the Gaussian special case of \eqref{LSI-intro}.  
 In 1966 E.~Nelson \cite{nelson66}, motivated by the problem of  semiboundedness of Hamiltonian operators 
in quantum field theory, proved the first version of the hypercontractivity   inequality of Theorem \ref{thmN}
with
dimension dependent bounds.
In order to encompass a larger  
class of Hamiltonians, J.~Glimm  \cite{glimm68} sharpened  Nelson's inequality in 1968 and removed
  the dimension dependence, thereby enabling   the inequality to work in infinite dimensions. 
Subsequently  Nelson  \cite{nelson-free-markov},  in 1973,  found the best hypercontractivity constants, which are those presented in Theorem \ref{thmN}.
 Pursuing a different track to semiboundedness of quantum field Hamiltonians,
 P.~Federbush \cite{federbush69},  showed in 1969 that  semiboundedness would follow from a
logarithmic Sobolev inequality  much more easily than from hypercontractivity. 
His semiboundedness theorem essentially asserts that a logarithmic Sobolev inequality implies semiboundedness.
In this paper he also gave a derivation 
of a  Gaussian  logarithmic Sobolev inequality using delicate Hermite function expansions in infinitely
 many variables. Although his version of a logarithmic Sobolev  inequality is not written down 
   in  this paper,
   it follows easily from 
   the identity  \cite[Equ.~(14)]{federbush69}, and  inequality \cite[Equ.~(21)]{federbush69}.
He thereby recovered  semiboundedness for the   class of Hamiltonians originally addressed by Nelson,
 though not the class encompassed by  Glimm's improvement. 
         Ironically, using the semiboundedness theorem of Federbush,  the sharp logarithmic Sobolev inequality of Stam would have 
 yielded semiboundedness of the large class addressed
  by  Glimm's improvement.
 But Stam's results were not known among this group of mathematical physicists 
till 1991, when Eric Carlen \cite{carlen91}, discovered Stam's paper and made  the connection
with the  Gaussian logarithmic Sobolev inequalities of the mathematical physics literature.
 In the meanwhile, the second author   \cite{gross75},  proved in 1975  that a family 
of hypercontractivity bounds,
  such as  those in Theorem \ref{thmN},  is  equivalent  to a logarithmic Sobolev inequality. 
 Best constants are preserved in this equivalence. Theorem \ref{thmequiv} is a typical case.
 He also
  proved the sharp form \eqref{LSI-intro} of the Gaussian logarithmic
  Sobolev inequality, which  Carlen later showed
  to be equivalent to the Euclidean form of Stam. 
  With the help of the equivalence theorem, one can understand better the relation between Stam's and Federbush's versions of the logarithmic Sobolev inequality: the former is equivalent to the strong form of Glimm,
   while the latter is equivalent  to the original form of Nelson.

Generalizations of the equivalence Theorem \ref{thmequiv} are now known to hold in a wide
variety of settings; see \cite{bakry-hypercontractivity-lnm,
  gross-lsi-survey-1992, gross-lsi-survey-2006} for surveys  and the recent exposition
  and historical background in  \cite{SimonPart3}.

Let us turn now to the complex setting; replace $\mathbb{R}^n$ by
$\mathbb{C}^n$ and {\len  suppose that} $\mu$ is standard Gaussian 
measure on  {\len $\mathbb{C}^n$}.
S.~Janson discovered in \cite{janson-hypercontractivity-1983} that if
one restricts the Ornstein--Uhlenbeck semigroup $e^{-tA}$ to the
holomorphic functions $\mathcal{H}$, then one obtains the property of
\defword{strong hypercontractivity}, in which the improvement in
integrability happens at earlier times:
\begin{theorem}
  For $0 < q \le p < \infty$, let $t_J(p,q) = \frac{1}{2}
  \log\left(\frac{p}{q}\right)$.  {\len Then, for}  any $t \ge t_J$, $e^{-tA}$
  is a contraction from $\mathcal{H} \cap L^q(\mu)$ to $\mathcal{H}
  \cap L^p(\mu)$.  
\end{theorem}
 {\len Several other proofs of this theorem} 
 followed \cite{carlen-integral-identities,
  zhou-contractivity-1991, janson-complex-hypercontractivity-1997}.
Note that ``Janson's time'' $t_J(p,q)$ is less than Nelson's time
$t_N(p,q)$ whenever $1 < q < p < \infty$. {\len  Moreover} Janson's strong
hypercontractivity also has content for $0 < q \le p \le 1$.  Very
roughly, the reason for this is that holomorphic functions are
harmonic, and so
 the second-order differential operator $A$,
  {\len when restricted to $\mathcal{H}$,} 
 reduces to the first-order operator $A f(z)
= {\len z\cdot}\nabla f(z) $.  Thus it is not surprising that its behavior
should be improved in this case.  We note for later reference that
in this case $A$ is the holomorphic vector {\len field} which {\len generates}
  the flow of the
dilations $\varphi_t(z) = tz$, meaning that the semigroup $e^{-tA}$ is
simply $e^{-tA} f(z) = f(e^{-t} z)$.

In the paper \cite{gross-hypercontractivity-complex}, the second author studied
{\len such Dirichlet form operators over}  
a complex Riemannian manifold $(M,g)$ equipped with a
smooth measure $\mu$, seeking to relate the logarithmic Sobolev
inequality to strong hypercontractivity {\len in a general holomorphic context}.  
The result was that the
former implies the latter, under fairly mild assumptions.  In this
result, the spaces $\mathcal{H} \cap L^p(\mu)$ must be replaced with
possibly smaller spaces denoted $\mathcal{H} L^p(\mu)$; see Remark
\ref{HLp-remark} below for the definitions used in
\cite{gross-hypercontractivity-complex}, and see
\cite{gross-hypercontractivity-complex} for a complete discussion of
the issues involved.
{\len  As in the Euclidean case}, the
Dirichlet form operator $A$ is given by the Laplacian {\len over}  $M$ plus a
complex vector field $Z$, so {\len that} on holomorphic functions {\len one has} 
$Af=Zf$.  If $Z$
is a holomorphic vector field, or equivalently, if the operator $A$
maps $\mathcal{H}$ into $\mathcal{H}$, we will say that $A$ is
\defword{holomorphic}. Let $Y = i(Z - \bar{Z})$ be the imaginary part
of $Z$.
\begin{theorem}\cite[Theorem 2.19]{gross-hypercontractivity-complex}\label{shc-holomorphic-case}
  Suppose that the operator $A$ is holomorphic and that the real vector
  field $Y$ is Killing.  If the logarithmic Sobolev inequality
  (\ref{LSI-intro}) holds, then for any $t \ge t_J(p,q)$, $e^{-tA}$ is
  a contraction from $\mathcal{H} L^q(\mu)$ to $\mathcal{H}
  L^p(\mu)$.  
\end{theorem}
A second proof was given in \cite{gross-strong-hypercontractivity},
which also allows for certain other types of boundary conditions in
the case that $(M,g)$ is incomplete.

{\nate The present paper is an extension of the results of
\cite{gross-hypercontractivity-complex,
  gross-strong-hypercontractivity}.  As noted, a key assumption of
those papers was that $A$ should be holomorphic.  This assumption is
in some sense natural, since it allows one to work entirely within the
holomorphic category; and it is satisfied by many interesting
examples.  But there are also many apparently innocuous settings in
which $A$ is not holomorphic.  See \cite{gross-qian,
  gross-hypercontractivity-complex, gross-strong-hypercontractivity}
and references therein for examples, counterexamples, and necessary
and sufficient conditions; the same condition is studied, in other
contexts, in \cite{calabi-extremal-kahler, futaki-kahler-einstein}.

To the best of our knowledge, until now, there have been no strong
hypercontractivity results akin to Theorem \ref{shc-holomorphic-case}
that apply in the case where $A$ is not holomorphic.  As such, our
goal here is to begin attacking this case by studying a particular
class of examples in which $A$ is not holomorphic, yet a strong
hypercontractivity theorem can still be proved. }

One possible way to approach
the case where $A$ is not holomorphic is, as suggested in
\cite[Section 7]{gross-strong-hypercontractivity}, to replace $A$ by
$B = P_{\mathcal{H}} A$, its $L^2$ orthogonal projection onto the
holomorphic functions $\mathcal{H}$.  Unfortunately, this does not
always work, and \cite{gross-strong-hypercontractivity} gives an
example of a complex manifold (a cylinder) for which $e^{-tB}$ is not
strongly hypercontractive, and is not even contractive on $L^p(\mu)$
for small $p<1$.

In the present paper, we examine a class of spaces in which the
operator $A$ is not holomorphic, and yet we are able to show that
$e^{-tB}$ is strongly hypercontractive, where $B$ is (at least on a
large class of functions) the holomorphic projection of $A$.  We work
in the setting of complex stratified Lie groups, where we replace the
Laplacian $\Delta$ by the hypoelliptic sub-Laplacian, and take as our
measure the corresponding hypoelliptic heat kernel.  A key observation
is that stratified Lie groups have a canonical dilation structure, and
it turns out that, as in the case of the Gaussian measure on
$\mathbb{C}^n$, the operator $B$ is essentially the holomorphic vector
field generated by dilations.

The paper is structured as follows.
\begin{itemize} \item  In Section \ref{sec:groups} we
introduce notation and review important properties of stratified
complex Lie groups $G$, their sub-Riemannian geometry, and the hypoelliptic
heat kernel $\rho_a$.  We also begin a discussion of holomorphic polynomials
on $G$. 
\item  Section \ref{sec:forms} defines the Dirichlet form $Q$ and
the operators $A,B$.  
\item In {\len  Section} \ref{sec:poly}, we study the density
properties of holomorphic polynomials, including an
orthogonal decomposition of holomorphic functions {\len  in $L^2(\rho_a)$} 
into homogeneous polynomials,
and obtain some additional properties of $A,B$ and their domains.
Section \ref{sec:poly} also defines the function spaces $\mathcal{H}
L^p(\rho_a)$ on which we work, and discusses related subtleties.
\item In Section \ref{sec:B}, we show that the operator $B$ is (up to
scaling and domain issues) identical to the holomorphic vector field
generated by dilations; we take advantage of this to show that
(except in trivial cases) the operator $A$ is not holomorphic.  
\item We
then proceed to show in Section \ref{sec:contractivity} that the
semigroup $e^{-tB}$ is a contraction on $L^p(\rho_a)$ for $0 < p <
\infty$; this is the special case of strong hypercontractivity with
$q=p$.  
\item Section \ref{sec:strong} contains the proof of our main
theorem, showing that if the logarithmic Sobolev inequality holds then
the semigroup $e^{-tB}$ is strongly hypercontractive.  
\item In Section
\ref{sec:heis} we specifically consider the complex Heisenberg group,
for which the logarithmic Sobolev inequality does indeed hold.
\end{itemize}

\section{Stratified complex groups}\label{sec:groups}

\subsection{Definitions}
In this section, we recall the definition of a stratified complex Lie
group (respectively, algebra) and its basic properties.  A
comprehensive reference on stratified Lie groups is \cite{blu-book}.

\begin{definition}
  Let $\mathfrak{g}$ be a finite-dimensional complex Lie algebra.  We
  say $\mathfrak{g}$ is \defword{stratified} of step $m$ if it admits a
  direct sum decomposition
  \begin{equation} \label{e.d2.1}
    \mathfrak{g} = \bigoplus_{j=1}^m V_j
  \end{equation}
  {\len for which}
  \begin{equation*}
    [V_1, V_j] = V_{j+1}, \qquad [V_1, V_m] = 0
  \end{equation*}
  and $V_m \ne 0$.
  A complex Lie group $G$ is \defword{stratified} if it is connected
  and simply connected and its Lie algebra $\mathfrak{g}$ is stratified.
\end{definition}

Using the Jacobi identity, it is easy to show that in a stratified Lie
algebra, we have $[V_k, V_j] \subset V_{j+k}$, where we take $V_{j+k}
= 0$ for $j+k > m$.  (Proceed by induction on $k$).  In particular,
$\mathfrak{g}$ is nilpotent of step { $m$}.  As such, the exponential
map $\exp : \mathfrak{g} \to G$ is a diffeomorphism, so we may as well
take $G=\mathfrak{g}$ as sets and let the exponential map be the
identity.  The group operation on $G$ can then be written down
explicitly using the Baker--Campbell--Hausdorff formula.  We note that
in $G$, the identity element $e$ is $0$, and the group inverse is
given by $g^{-1} = -g$.  We shall use $L_x : G \to G$ to denote the
left translation map $L_x(y)= x \cdot y$.  We identify $\mathfrak{g}$
with the tangent space $T_e G$, and for $\xi \in \mathfrak{g}$,
$\widetilde{\xi}$ is the left-invariant vector field on $G$ with
$\widetilde{\xi}(e) = \xi$.

{ Since $\mathfrak{g}$ is a finite-dimensional vector space, it carries
a translation-invariant Lebesgue measure, which is unique up to
scaling. We fix one such measure and denote it by $m$; integrals with
respect to $dx, dy$, etc., will also be understood to refer to this
measure.  Then $m$ is also a measure on $G$.  It is easy to verify
that $m$ is bi-invariant under the group operation on $G$, so $m$ is
(again up to scaling) the Haar measure on $G$.}

\begin{notation}\label{n.d2.16}
  We define \textbf{convolution} on $G$ by
\begin{equation}
(f\ast g)(x)=\int_{G}f(xy^{-1})g(y)dy=\int_{G}f(z)g(z^{-1}x)dz
\end{equation}
when the Lebesgue integral exists.
\end{notation}

Our motivating examples are the complex Heisenberg and
Heisenberg--Weyl groups.

\begin{example}\label{ex-heis}
  The \defword{complex Heisenberg Lie algebra} is the complex Lie
  algebra $\mathfrak{h}_3^{\mathbb{C}}$ given by $\mathbb{C}^3$ with
  the bracket
  defined by
  \begin{equation}\label{heis-bracket}
    [(z_1, z_2, z_3), (z_1', z_2', z_3')] = (0, 0, z_1 z_2' - z_1' z_2).
  \end{equation}
  Taking $V_1 = \{(z_1, z_2, 0) : z_1, z_2 \in \mathbb{C}\}$ and $V_2
  = \{(0,0,z_3) : z_3 \in \mathbb{C} \}$, it is clear that
  $\mathfrak{h}_3^{\mathbb{C}}$ is stratified of step 2.  The
  \defword{complex Heisenberg group} $\mathbb{H}_3^{\mathbb{C}}$ is
  then $\mathbb{C}^3$ with the group operation $g \cdot h =
  g+h+\frac{1}{2}[g,h]$, which we may write in coordinates as
  \begin{equation*}
    (z_1, z_2, z_3) \cdot (z_1', z_2', z_3') = (z_1+z_1', z_2+z_2',
z_3+z_3'+\frac{1}{2}(z_1 z_2' - z_2 z_1')).
  \end{equation*}
\end{example}

Some readers may be used to seeing the Heisenberg group as the group
of upper triangular matrices with $1$s on the diagonal.  Let us note that
by mapping the element $(z_1, z_2, z_3) \in
\mathbb{H}_3^{\mathbb{C}}$ to the matrix
\begin{equation*}
  \begin{pmatrix}
    1 & z_1 & z_3 + \frac{1}{2} z_1 { z_2} \\ 0 & 1 & z_2 \\ 0 & 0 & 1
  \end{pmatrix}
\end{equation*}
we have an embedding of the Lie group $\mathbb{H}_3^{\mathbb{C}}$ into
the Lie group $GL(\mathbb{C}, 3)$ of invertible $3 \times 3$ complex
matrices, whose image is precisely the upper triangular matrices with
$1$s on the diagonal.  So this realization of the complex Heisenberg
group is isomorphic to ours.  (Note that the slightly strange-looking
upper right entry of the matrix above is chosen so that this map is a
group homomorphism.)

\begin{example}\label{ex-heis-weyl}
  Generalizing the previous example, the \defword{complex
    Heisenberg--Weyl Lie algebra} of dimension $2n+1$ is the complex Lie
  algebra $\mathfrak{h}_{2n+1}^{\mathbb{C}}$ given by $\mathbb{C}^{2n+1}$ with
  the bracket
  defined by
  \begin{equation}
    [(z_1, \dots, z_{2n+1}), (z_1', \dots, z_{2n+1}')] = \left(0, \dots, 0,
    \sum_{k=1}^n z_{2k-1} z_{2k}' - z_{2k-1}' z_{2k}\right).
  \end{equation}
  This again is stratified of step 2, taking $V_1 = \{(z_1, \dots,
  z_{2n}, 0) : z_1, \dots, z_{2n} \in \mathbb{C}\}$ and $V_2
  = \{(0,\dots, 0 ,z_{2n+1}) : z_{2n+1} \in \mathbb{C} \}$.  The
  \defword{complex Heisenberg--Weyl group}
  $\mathbb{H}_{2n+1}^{\mathbb{C}}$ is again $\mathbb{C}^{2n+1}$ with the
  group operation $g \cdot h = g+h+\frac{1}{2}[g,h]$.  
\end{example}

\subsection{The dilation semigroup}

\begin{definition}
  For $\lambda \in \mathbb{C}$, the \defword{dilation map} on
  $\mathfrak{g}$ or $G$ is defined by
  \begin{equation}
    \delta_{\lambda}(v_{1}+\cdots+v_{m})=\sum_{k=1}^{m}\lambda^{k}v_{k}\qquad
    v_{j}\in V_{j},\quad j=1,\ldots,m. \label{e.d2.4}
  \end{equation}
\end{definition}

It is straightforward to verify that for $\lambda \ne 0$,
$\delta_\lambda$ is an algebra automorphism of $\mathfrak{g}$ and a
group automorphism of $G$, and that
\begin{equation}
  \delta_{\lambda\mu}=\delta_{\lambda} \circ \delta_{\mu}\qquad\lambda,\mu
  \in\mathbb{C.} \label{e.d2.5}%
\end{equation}
Moreover, $\delta_\lambda$ is linear, so the derivative at the
identity of $\delta_\lambda : G \to G$ is $(\delta_\lambda)_* =
\delta_\lambda : \mathfrak{g} \to \mathfrak{g}$.

We note that $\delta_{\lambda}$ scales the Lebesgue measure $m$ by
\begin{equation}\label{m-dilate}
m(\delta_\lambda(A)) = |\lambda|^{2D} m(A),
\end{equation} 
where 
{\len $D : = \sum_{j=1}^m j \dim_{\mathbb{C}} V_j$} is the homogeneous dimension of $G$.  Thus
for an integrable function $f$, we have
\begin{equation}\label{m-dilate-integrate}
  \int_G f \circ \delta_\lambda\,dm = |\lambda|^{-2D} \int_G f\,dm.
\end{equation}

We can then consider the vector fields generating this semigroup.

\begin{definition}
  We define the real vector fields $X,Y$ on $G$ as
  \begin{align}
    (Xf)(z) &= \frac{d}{ds}\Big|_{s=0}f(\delta_{e^{s}}z)\qquad f\in C^{\infty}(G),
    \label{e.d2.6} \\
    (Yf)(z) &= \frac{d}{d\theta}\Big|_{\theta=0}f(\delta_{e^{i\theta}}z)\qquad f\in
    C^{\infty}(G) \label{e.d2.7}
  \end{align}
  and the complex vector field $Z$ by
  \begin{equation}
    Z=\frac{1}{2}(X-iY). \label{e.d2.8}
  \end{equation}
\end{definition}

\begin{remark}\label{i-remark}
  To {\len  remind the reader of standard conventions,}
  we note that the $i$ appearing in
  (\ref{e.d2.8}) does not denote the complex structure on
  $\mathfrak{g}$, but rather ordinary scalar multiplication for
  complex vector fields.  Formally, $Z$ is a smooth section of the
  complexified tangent bundle $TG \otimes_{\mathbb{R}} \mathbb{C}$,
  which has a natural complex vector space structure with scalar
  multiplication { $\zeta \cdot (v_x \otimes \eta) = v_x
    \otimes (\zeta \eta)$},
  and in which $TG$ embeds naturally via $v_x \mapsto v_x
  \otimes 1$.  
\end{remark}

\begin{lemma}
\label{l.d2.5} $Z$ is a holomorphic vector field of type $(1,0)$.
\end{lemma}

\begin{proof}
Let $z_{1},\ldots,z_{N}$ be complex coordinates on $G\equiv\mathfrak{g}$
relative to a basis of $\mathfrak{g}$ adapted to the decomposition in
(\ref{e.d2.1}). Then $\delta_{\lambda}z=(\ldots,\lambda^{c_{j}}z_{j},\ldots)$ for
positive integers $c_{1},\ldots,c_{N}$. Hence for any function $f\in
C^{\infty}(G)$ we have
\[
(Xf)(z)=\sum_{j=1}^{N}\left\{  c_{j}z_{j}\frac{\partial f}{\partial z_{j}%
}+c_{j}\overline{z}_{j}\frac{\partial f}{\partial\overline{z}_{j}}\right\}
\]
and
\[
(Yf)(z)=\sum_{j=1}^{N}\left\{ ic_{j}z_{j}\frac{\partial f}{\partial z_{j}}%
-ic_{j}\overline{z}_{j}\frac{\partial f}{\partial\overline{z}_{j}} \right\}.
\]
Thus
\begin{equation}
(Zf)(z)=\sum_{j=1}^{N}c_{j}z_{j}\frac{\partial f}{\partial z_{j}}.
\label{e.d2.9}%
\end{equation}
\end{proof}

\subsection{Holomorphic polynomials and Taylor series}

\begin{notation}
  $\mathcal{H}$ denotes the vector space of holomorphic functions on $G$.
\end{notation}

The dilations $\delta_{\lambda}$ on $G$ lead naturally to a notion of
homogeneous functions and polynomials on $G$.  These functions were
used extensively in \cite{folland-stein}, in the context of real
homogeneous groups.  For us, they will be used as a convenient class
of holomorphic test functions.  In this section, we define these
functions and verify a few key properties that will be important in
this paper.

\begin{definition}
\label{d.d2.6}  Let $k$ be a nonnegative integer. A function $f:G\rightarrow
\mathbb{C}$ is \defword{homogeneous} of degree $k$ if
\begin{equation}
f(\delta_{\lambda}z)=\lambda^{k}f(z)\ \text{for all}\ z\in G\ \text{and}%
\ 0\neq\lambda\in\mathbb{C}. \label{e.d2.10}%
\end{equation}
\end{definition}

\begin{example}
\label{ex.d2.7}If $G$ is the complex Heisenberg group with complex coordinates
$z_{1}, z_{2}, z_{3}$ then $z_{1}^{2}, z_{1}z_{2}, z_{2}^{2}, z_{3}$ are all homogeneous of degree $2$.
\end{example}

Note that if $f$ is homogeneous of degree $k$ then (\ref{e.d2.10}) and
(\ref{e.d2.6}), (\ref{e.d2.7}), (\ref{e.d2.8}) give
\begin{align}
Xf(z)&=kf(z),                                              \label{e.d2.11} \\
(Yf)(z)&=ikf(z)                                            \label{e.d2.12} \\
\intertext{and}
(Zf)(z)&=kf(z).                                             \label{e.d2.13}%
\end{align}

\begin{notation} \label{n.d2.8}
  For $k=0,1,2,\ldots$ we will denote by $\mathcal{P}_{k}$
  the set of all holomorphic  functions on $G$ which are
  homogeneous of degree $k$.
\end{notation}

\begin{lemma} \label{l.d2.9}
  Every holomorphic function $f \in \mathcal{H}$ has a unique decomposition
  of the form
  \begin{equation}
    f(z) = \sum^{\infty}_{k=0} f_{k}, \qquad f_{k} \in\mathcal{P}_{k} \label{e.d2.14}%
  \end{equation}
  in the sense of pointwise convergence.
\end{lemma}

\begin{proof}
  Notice first that the function $G \times \mathbb{C} \ni (z,\lambda)
  \mapsto \delta_{\lambda} z \in G$ is holomorphic in the sense that each coordinate of
$\delta_{\lambda}z$, in the basis used in Lemma \ref{l.d2.5}, is
  holomorphic.

  Suppose $f:G\rightarrow\mathbb{C}$ is holomorphic, so that
  { $(z,\lambda) \mapsto f(\delta_{\lambda}z)$}
  is holomorphic on $G\times\mathbb{C}$.  Fix
  an arbitrary $z \in G$.  Then the function
  $u(\lambda):=f(\delta_{\lambda}z)$ is an entire function on
  $\mathbb{C}$, and its Taylor expansion
\begin{equation}\label{u-def}
u(\lambda)=\sum_{n=0}^{\infty}\lambda^{n}a_{n}(z)
\end{equation}
determines functions $a_{n}(z)$ which are holomorphic functions on $G$ because
\begin{equation*}
a_{n}(z)=\frac{1}{n!} \frac{d^{n}}{d\lambda^{n}}\Big|_{\lambda=0}
f(\delta_{\lambda}z).
\end{equation*}
Now if $\mu\in\mathbb{C}$ then
\[
\sum_{n=0}^{\infty}\lambda^{n}a_{n}(\delta_{\mu}z)=f(\delta_{\lambda}%
\delta_{\mu}z)=f(\delta_{\lambda\mu}z)=\sum_{n=0}^{\infty}\ (\lambda\mu
)^{n}a_{n}(z)\ \text{for all}\ \lambda\in\mathbb{C.}%
\]
Hence
\[
a_{n}(\delta_{\mu}z)=\mu^{n}a_{n}(z)\ \text{for all}\ z\in G.
\]
Therefore $a_{n}\in\mathcal{P}_{n}$. This proves the existence of the
functions $f_{k}$ satisfying (\ref{e.d2.14}). If $\{g_{k}\}$ is another set
satisfying (\ref{e.d2.14}) then
\[
\sum_{k=0}^{\infty}\lambda^{k}g_{k}(z)=\sum_{k=0}^{\infty}g_{k}(\delta
_{\lambda}z)=f(\delta_{\lambda}z)=\sum_{k=0}^{\infty}f_{k}(\delta_{\lambda
}z)=\sum_{k=0}^{\infty}\lambda^{k}f_{k}(z)
\]
for all $\lambda\in\mathbb{C}$. Hence $g_{k}(z)=f_{k}(z)$ for all $k$ and $z.$
\end{proof}

\begin{notation}
{ Let $\mathcal{P}$ denote the linear span of $\{\mathcal{P}_{k}: k\ge
0\}$; i.e. the set of all \emph{finite} sums of homogeneous functions
(of possibly different degrees).}
\end{notation}

\begin{lemma}
\label{r.d2.10} $\mathcal{P}$ is the set of holomorphic polynomials.
\end{lemma}

\begin{proof}
  In the adapted coordinates $z_1,\dots, z_N$, a monomial $\prod_{j=1}^N
  z_j^{k_j}$ is homogeneous of degree $\sum_{j=1}^N k_j
  c_j$. Therefore any holomorphic polynomial lies in
  $\mathcal{P}$. Conversely, we need to show that a function
  $f\in\mathcal{P}_k$ is actually a polynomial. If its power series
  expansion is given by
  \begin{align}
    f(z) = \sum_{k_1,\dots, k_N \ge 0} a_{k_1,\dots, k_N} z^{k_1} \cdots z^{k_N}
  \end{align}
  then, for all complex $\lambda \ne 0$, we have
  \begin{align}
    \lambda^k f(z) = f(\delta_\lambda z) 
    = \sum_{k_1,\dots, k_N} a_{k_1,\dots, k_N} z^{k_1} \cdots z^{k_N} 
    \lambda^{\sum_{j=1}^N k_j c_j}.
  \end{align}
  Since the coefficient of $\lambda^r$ on the right must be zero  for all $z$ if $ r \ne k$ 
  we actually have
  \begin{align}
    f(z) = \sum_{\sum_{j=1}^N k_j c_j = k}
    a_{k_1,\dots, k_N} z^{k_1} \cdots z^{k_N}
  \end{align}
  The subscripts in the sum form a finite set, showing that $f$ is a polynomial.
\end{proof}

\begin{corollary}\label{Pk-finite-dim}
  $\mathcal{P}_k$ is finite dimensional.
\end{corollary}

\begin{lemma}
\label{l.d2.11} If $f$ is holomorphic and is given by (\ref{e.d2.14}) then
\begin{equation}
(Zf)(z)=\sum_{k=0}^{\infty}kf_{k}(z). \label{e.d2.15}%
\end{equation}

\end{lemma}

\begin{proof}
Since $f(\delta_{\lambda}z)=\sum_{k=0}^{\infty}\lambda^{k}f_{k}(z)$ for all
$\lambda\in\mathbb{C}$ we have
\begin{equation*}
(Zf)(z)=(Xf)(z)=\frac{d}{ds}\big|_{s=0}\sum_{k=0}^{\infty}e^{ks}f_{k}(z)=\sum
_{k=0}^{\infty}kf_{k}(z).
\end{equation*}
The interchange of derivative and sum is justified since
$\sum_{k=0}^{\infty}e^{ks}f_{k}(z)$ is the Taylor series of the
holomorphic function $u(e^s)$, where $u(\lambda) := f(\delta_\lambda
z)$ as in the proof of Lemma \ref{l.d2.9}, and this can be
differentiated termwise.
\end{proof}

We remark for future reference that by (\ref{e.d2.11}) and
(\ref{e.d2.13}), we have
\begin{equation}
  \label{Z-X-poly}
  Zf = Xf, \qquad f \in \mathcal{P}.
\end{equation}

\begin{lemma}\label{dilate-poly}
  Let $\xi \in V_j$  and $f \in \mathcal{P}_k$.  Then $\widetilde{\xi} f
  \in \mathcal{P}_{k-j}$ if $k \ge j$, and $\widetilde{\xi} f = 0$ if
  $k < j$.
\end{lemma}

\begin{proof}
  First, since $f$ is holomorphic and $\widetilde{\xi}$ is left-invariant, $\widetilde{\xi}
  f$ is holomorphic.
  Next, since $\delta_\lambda$ is a group homomorphism, for any $z \in
  G$ we
  have $L_{\delta_\lambda(z)} = \delta_\lambda \circ L_z \circ
  \delta_{\lambda^{-1}}$.  By left-invariance of $\widetilde{\xi}$
  we have
  \begin{align*}
    (\widetilde{\xi} f)(\delta_\lambda z) &=
    ((L_{\delta_\lambda(z)})_* \xi) f \\
    &= ( \delta_\lambda (L_z)_* \delta_{\lambda^{-1}} \xi ) f \\
    &= \lambda^{-j} (\delta_\lambda (L_z)_* \xi) f && \text{since $\xi
      \in V_j$}\\
    &= \lambda^{-j} ((L_z)_* \xi) (f \circ \delta_\lambda) \\
    &= \lambda^{k-j} ((L_z)_* \xi) f && \text{since $f \in \mathcal{P}_k$}\\
    &= \lambda^{k-j} \widetilde{\xi} f(z).
  \end{align*}
  Thus $f \in \mathcal{P}_{k-j}$.  If $k-j < 0$ then the fact that
  {$\widetilde{\xi} f$}
   is continuous at the identity leads to the
  conclusion that $f \equiv 0$.
\end{proof}

\subsection{Sub-Riemannian geometry on $G$}\label{sub-riemannian-section}

As before, let $\mathfrak{g}$ be a stratified complex Lie algebra with
its connected, simply connected complex Lie group $G$.  For this
section, we will use $J$ to denote the complex structure on
$\mathfrak{g}$.  In this section, we collect a number of facts about
the { sub-Riemannian} geometry of $G$ and its
hypoelliptic Laplacian.  { Although much of this
  development is standard, we shall be rather explicit with our
  definitions to fix notation and avoid any possible ambiguity.}

View $\mathfrak{g}$ as a real vector space, and let $\mathfrak{g}^*$
be its dual space.  Let $h : \mathfrak{g}^* \times \mathfrak{g}^* \to
\mathbb{R}$ be a symmetric, positive semidefinite, real bilinear form
on $\mathfrak{g}^*$. { We shall think of $h$ as a ``dual
  metric'' on the dual $\mathfrak{g}^*$, despite the fact that it is
  degenerate, i.e. only positive semidefinite instead of positive definite.}  Suppose further that
$h$ is Hermitian, i.e. $h(J^* \alpha, J^* \beta) = h(\alpha, \beta)$,
where $J^*$ is the adjoint of $J$.  { (This ensures that
  $h$, in some sense, respects the complex structure of
  $\mathfrak{g}$.)}

Let $K := \{ \alpha \in \mathfrak{g}^* : h(\alpha,\alpha) = 0\}$ be
the null space of $h$ and let $H = K^0 = \bigcap_{\alpha \in K} \ker
\alpha \subset \mathfrak{g}$ be the backward annihilator of $K$; $H$
is called the \defword{horizontal subspace} of $\mathfrak{g}$.  Note
that $H$ is invariant under $J$.

Henceforth we assume the following non-degeneracy condition:
\begin{assumption}
  $H = V_1$.
\end{assumption}

In particular, H\"ormander's condition is satisfied: $H$ generates
$\mathfrak{g}$.  In fact, H\"ormander's condition is satisfied if and
only if $V_1 \subset H$; we need the opposite inclusion to ensure that
$h$ interacts nicely with the dilation structure on $G$.

Now $h$ induces a natural real-linear map $\Phi : \mathfrak{g}^* \to
\mathfrak{g}$ defined by $\alpha(\Phi \beta) = h(\alpha, \beta)$ with
kernel $K$ and image $H$.  (Note that $\Phi = J \Phi J^*$.)  We may
then define a bilinear form $g : H \times H \to \mathbb{R}$ on $H$ by
$g(\Phi \alpha, \Phi \beta) = h(\alpha, \beta)$ which is easily seen
to be well-defined, Hermitian (i.e. $g(v,w) = g(Jv, Jw)$), and
positive definite.

{
By left translation, we can extend $h$ to a (degenerate)
left-invariant dual metric (still denoted by $h$) on $T^* G$, defined
by $h_x(\alpha_x, \beta_x) = h(L_x^* \alpha_x, L_x^* \beta_x)$ for
$\alpha_x, \beta_x \in T^*_x G$.  Then $H$ extends to a left-invariant
sub-bundle of $TG$; namely, $v_x \in H_x \subset T_x G$ iff
$(L_{x^{-1}})_* v_x \in H$, which happens iff $\alpha_x(v_x) = 0$ for
every $\alpha_x \in T^*_x G$ satisfying $h_x(\alpha_x, \alpha_x)=0$.
$H_x$ is the \defword{horizontal subspace} of $T_x G$ and vectors $v_x
\in H_x$ are said to be \defword{horizontal}.  The bundle $H$ itself
is sometimes called the \defword{horizontal distribution}.  We can
also extend $g$ to a left-invariant positive definite inner product on
$H$, defined by $g_x(v_x, w_x) = g((L_{x^{-1}})_* v_x, (L_{x^{-1}})_*
w_x)$ for $v_x, w_x \in H_x$.  $g$ is called a \defword{sub-Riemannian
  metric}.  If we define $\Phi_x : T^*_x G \to T_x G$ by $\Phi_x =
(L_x)_* \Phi L_x^*$ then the image of $\Phi_x$ is $H_x$, and we have
$g_x(\Phi_x \alpha_x, \Phi_x \beta_x) = h_x(\alpha_x, \beta_x)$.
Given a smooth function $f : G \to \mathbb{R}$, we can define its
left-invariant \defword{sub-gradient} $\nabla f \in H$ by $\nabla f(x)
= \Phi_x(df(x))$.

We wish to consider complex functions, one-forms, vector fields, etc.,
on $G$, so we shall now complexify everything in sight.  At each $x
\in G$, we form the complexified tangent space $T_x G \otimes
\mathbb{C}$, which, as mentioned in Remark \ref{i-remark}, is a
complex vector space with the complex scalar multiplication $\zeta
\cdot (v_x \otimes \eta) = v_x \otimes (\zeta \eta)$.  When taking
this tensor product, we view $T_x G$ as a \emph{real} vector space,
forgetting that it already has the natural complex structure $J_x =
(L_x)_* J (L_{x^{-1}})_*$.  This means that $T_x G \otimes \mathbb{C}$
now has \emph{two} distinct complex structures: multiplication by $i$
(i.e. $v_x \otimes \eta \mapsto v_x \otimes i \eta$), and $J_x$ (which
we extend to $T_x G \otimes \mathbb{C}$ by complex linearity: $J_x i
v_x = i J_x v_x$).  A complex vector field can thus be viewed as a
smooth section of the complexified tangent bundle $TG \otimes
\mathbb{C}$.  The complexified horizontal bundle $H \otimes
\mathbb{C}$ is naturally contained in $TG \otimes \mathbb{C}$.  We
likewise form the complexified cotangent space $T^*_x G \otimes
\mathbb{C}$ and note that it can be viewed as the complex dual space
of $T_x G \otimes \mathbb{C}$.  If $f : G \to \mathbb{C}$ is a complex
function, written as $f=u+iv$, then its differential $df$ is a complex
one-form, a smooth section of $T^* G \otimes \mathbb{C}$ given by $df
=  du + i dv$.  $T^* G \otimes \mathbb{C}$ also has two complex structures:
multiplication by $i$, and $J^*_x = L_{x^{-1}}^* J^* L_x^*$ (extended
by complex linearity).  In particular, if $f$ is holomorphic then we
have the Cauchy--Riemann equation $J^* df = i df$; that is, $df$ is a
complex one-form of type $(1,0)$.

Now we extend $h$ to $T^* G \otimes \mathbb{C}$ in such a way as to
make it \emph{complex bilinear} with respect to multiplication by $i$;
that is, $h_x(i \alpha_x, \beta_x) = h_x(\alpha_x, i\beta_x) = i
h_x(\alpha_x, \beta_x)$.  So now $h_x$ is complex bilinear with
respect to $i$, and Hermitian with respect to $J_x^*$.  We likewise
extend $\Phi_x$ to a complex linear map $\Phi_x : T^*_x G \otimes
\mathbb{C} \to H_x \otimes \mathbb{C}$, and then defining $g_x$
analogously as before makes it a complex bilinear form on $H_x \otimes
\mathbb{C}$.  Note that $g_x$ remains Hermitian with respect to $J_x$.  By an
abuse of terminology, we shall continue to call $g$ and $h$ the
sub-Riemannian metric and dual metric, respectively.  We now also have
the sub-gradient $\nabla f(x) = \Phi_x(df(x)) \in T_x G \otimes
\mathbb{C}$ defined for complex functions.

We can describe this geometry more explicitly by choosing a set of
left-invariant real vector fields $X_1, Y_1, \dots, X_k, Y_k$ which
span $H$, are $g$-orthonormal, and have $Y_j = J X_j$.  Then the
sub-gradient is given by
\begin{equation*}
  \nabla f(x) = \sum_j (X_j f)(x) X_j(x) +
  (Y_j f)(x) Y_j(x)
\end{equation*}
and for smooth $f_1, f_2 : G \to \mathbb{C}$ we
have }
\begin{equation}\label{gXY}
g(\nabla f_1, \nabla \bar{f}_2) = h(df_1, d\bar{f}_2) =
\sum_j  \{X_j f_1 X_j \bar{f}_2 + Y_j f_1 Y_j \bar{f}_2\}.
\end{equation}
We shall use $|\nabla f|^2$ as shorthand for $g(\nabla f, \nabla
\bar{f})$.

Alternatively, letting 
\begin{equation}\label{Zj-def}
\begin{split}
  Z_j &= \frac{1}{2}(X_j - i Y_j) \\
  \bar{Z}_j &= \frac{1}{2} (X_j + i Y_j)
\end{split}
\end{equation}
 so that $Z_j$ and $\bar{Z}_j$ are { complex vector fields} of type
$(1,0)$ and $(0,1)$ respectively, we get 
\begin{align}
\nabla f(x) &= 2 \sum_j  \Big((Z_j f)(x)
\bar{Z}_j(x) + (\bar{Z}_j f)(x) Z_j(x)\Big) \\
g(\nabla f_1, \nabla
\bar{f}_2) = h(df_1, d \bar{f}_2) &= 2 \sum_j  \Big(Z_j f_1 \bar{Z}_j \bar{f}_2 + \bar{Z}_j f_1 Z_j \bar{f}_2 \Big).
\label{gZ}
\end{align}
We remark in passing that $X_j$ and $Y_j$ commute (since, using that
fact that $\mathfrak{g}$ is a complex Lie algebra, $[X_j, Y_j] = [X_j,
  JX_j] = J[X_j, X_j] = 0$), and thus $Z_j$ and $\bar{Z}_j$ commute.

Note that when $f$ is real, we have
\begin{equation}\label{gradsq-real}
  |\nabla f|^2 := g(\nabla f, \nabla f) = h(df, df) = 4
  \sum_j |Z_j f|^2
\end{equation}
and when $f$ is holomorphic,
\begin{equation}\label{gradsq-holo}
  |\nabla f|^2 = 2
  \sum_j |Z_j f|^2.
\end{equation}

\begin{example}\label{ex-heis-metric}
  Returning to the example of the complex Heisenberg group begun in
  Example \ref{ex-heis}, consider $\mathbb{H}_3^{\mathbb{C}} =
  \mathbb{C}^3$ with its Euclidean coordinates $(z_1, z_2, z_3)$. Let
  $h$ be the left-invariant dual metric
  given at the identity $e=0$ by
  \begin{align*}
    h_e(dz_1, d\bar{z}_1) = h_e(dz_2, d\bar{z}_2) &= 2 \\
    h_e(dz_3, d\bar{z}_3) &= 0 \\
    h_e(dz_j, d\bar{z}_k) &= 0, \qquad j \ne k.
  \end{align*}
  This
  makes $h$ Hermitian with respect to the complex structure of
  $\mathbb{H}_3^{\mathbb{C}}$, so that $h_e(dz_j, dz_k) =
  h_e(d\bar{z}_j, d\bar{z}_k) = 0$ for all $j,k$.  (The 2 appearing in
  the first line ensures that the cotangent vectors $dx_i, dy_j$ are
  orthonormal under $h_e$.)

  From now on, any occurrence of $\mathbb{H}_3^{\mathbb{C}}$ will be
  understood to carry this dual metric $h$, and the corresponding
  metric $g$.

  We can choose the left-invariant complex vector fields $Z_j$
  discussed in (\ref{Zj-def}) to be those which equal
  $\frac{\partial}{\partial z_j}$ at the identity.  They are given by
  \begin{align*}
    Z_1 &= \frac{\partial}{\partial z_1} 
    - \frac{1}{2} z_2 \frac{\partial}{\partial z_3} \\
    Z_2 &= \frac{\partial}{\partial z_2} 
    + \frac{1}{2} z_1 \frac{\partial}{\partial z_3} \\
    Z_3 &= \frac{\partial}{\partial z_3}.
  \end{align*}
\end{example}

\begin{example}\label{ex-heis-weyl-metric}
  For the Heisenberg--Weyl group $\mathbb{H}^{\mathbb{C}}_{2n+1}$ of
  Example \ref{ex-heis-weyl}, we may similarly define a
  left-invariant dual metric $h$ by
  \begin{align*}
    h_e(dz_j, d\bar{z}_j) &= 2, & 1 \le j \le 2n \\
    h_e(dz_{2n+1}, d\bar{z}_{2n+1}) &= 0 \\
    h_e(dz_j, d\bar{z}_k) &= 0, & j \ne k.
  \end{align*}
\end{example}

Let us see how the dilations interact with the left-invariant real vector fields $X_j,
Y_j$.  If $y \in G$, and
$\lambda = \alpha + i \beta \in \mathbb{C}$, we have
\begin{equation}\label{X-dilation}
  \begin{split}
  (\delta_\lambda)_* X_j(y) &= (\delta_\lambda L_y)_* X_j(e) \\
  &= (L_{\delta_\lambda(y)} \delta_\lambda)_* X_j(e) \\
  &= (L_{\delta_\lambda(y)})_* (\alpha X_j(e) + \beta J X_j(e)) \\
  &= \alpha X_j(\delta_\lambda(y)) + \beta J X_j(\delta_\lambda(y)).
\end{split}
\end{equation}
The same holds for $Y_j$.  Thus we get
\begin{equation}\label{Z-dilation}
  \begin{split}
    (\delta_\lambda)_* Z_j(y) &= \lambda Z_j(\delta_\lambda(y)) \\
    (\delta_\lambda)_* \bar{Z}_j(y) &= \bar{\lambda}
    \bar{Z}_j(\delta_\lambda(y)).
  \end{split}
\end{equation}

The \defword{sub-Laplacian} $\Delta$ is defined by
\begin{equation}\label{sub-Laplacian-def}
  \Delta  = \sum_j X_j^2  + Y_j^2 = 4 \sum_j Z_j \bar{Z}_j.
\end{equation}
It is shown in
\cite{purplebook} that $\Delta$, with domain $C^\infty_c(G)$, is a
hypoelliptic operator and is essentially self-adjoint on $L^2(m)$.
As a consequence of (\ref{Z-dilation}), we have 
\begin{equation}
  \Delta (f \circ \delta_\lambda) = |\lambda|^2 (\Delta f) \circ \delta_\lambda.
\end{equation}
Likewise, if $e^{s \Delta/4}$ is the heat semigroup for $\Delta$, we
have
\begin{equation}\label{semigroup-dilate}
  e^{s \Delta/4} (f \circ \delta_\lambda) = (e^{s |\lambda|^2
    \Delta/4} f) \circ \delta_\lambda.
\end{equation}

Finally,  we recall the definition of the Carnot--Carath\'eodory
  distance on $G$, { and some of its basic properties}.  Suppose $\gamma : [0,1] \to G$ is a smooth path.
If $\dot{\gamma}(t) \in H_{\gamma(t)}$ for each $t$, we say $\gamma$
is \defword{horizontal}, and we define its \defword{length} by 
\begin{equation}
  \ell(\gamma) = \int_0^1 \sqrt{g(\dot{\gamma}(t), \dot{\gamma}(t))}\,dt.
\end{equation}
Then for $x,y \in G$, we define the \defword{Carnot--Carath\'eodory
  distance} $d$ by
\begin{equation*}
  d(x,y) = \inf\{\ell(\gamma) : \gamma \text{ horizontal},
  \gamma(0)=x, \gamma(1)=y\}
\end{equation*}
Since H\"ormander's condition is satisfied, the Chow--Rashevskii and
ball-box theorems \cite{montgomery,nagel-stein-wainger} imply that $d(x,y) <
\infty$, and that $d$ is a metric which induces the manifold topology
on $G$ (which indeed is just the Euclidean topology on the
finite-dimensional vector space $G = \mathfrak{g}$).

Since we are denoting the complex structure on $\mathfrak{g}$ by $J$,
for $v \in V_1 \subset \mathfrak{g} = T_e G$ we have $(\delta_{\alpha
  + i \beta})_* v = \delta_{\alpha + i \beta}(v) = \alpha v + \beta
Jv$.  Thus, for $v,w \in V_1$ we have $g((\delta_\lambda)_* v,
(\delta_\lambda)_* w) = |\lambda|^2 g(v,w)$.  Since $\delta_\lambda$
is a group homomorphism and $g$ is left invariant, it follows that the
same holds for $v,w \in H_x$.  In particular,
$\ell(\delta_\lambda(\gamma)) = |\lambda| \ell(\gamma)$, and so $d(e,
\delta_\lambda(x)) = |\lambda| d(e,x)$.

By fixing a basis for $\mathfrak{g}$, we may linearly identify it
(non-canonically) with Euclidean space $\mathbb{R}^{\dim_{\mathbb{R}}
  \mathfrak{g}}$; let $|\cdot|$ denote the pullback of the Euclidean
norm onto $\mathfrak{g}$.  For $v \in \mathfrak{g}$, write $v = v_1 + \dots + v_m$ with $v_k
\in V_k$ and let
\begin{equation}
  |v|_1 = \sum_{k=1}^m |v_k|^{1/k}.
\end{equation}
Note that $|\delta_{\lambda} v|_1 = |\lambda| |v|_1$.  Since we have
identified $G$ with $\mathfrak{g}$ as a set, $|\cdot|_1$ also makes
sense on $G$.  It is shown in \cite[Proposition 5.1.4]{blu-book} that
there is a constant $c$ such that for every $x \in G$ we have
\begin{equation}\label{d-scaling}
  \frac{1}{c} |x|_1 \le d(e,x) \le c |x|_1.
\end{equation}
The proof is simple: since $d(e,\cdot)$ and $|\cdot|_1$ have the same
scaling with $\delta_\lambda$, it suffices to consider $x$ with $|x|_1
= 1$.  The set of such $x$ is compact, so $d(e, \cdot)$ attains a
finite maximum and a nonzero minimum on this set.

\subsection{Properties of the heat kernel}\label{subsec:heatkernel}

It is shown in \cite{purplebook} that the Markovian heat semigroup
$e^{s \Delta/4}$ admits a right convolution
kernel $\rho_s$, i.e. $e^{s \Delta/4} f = f \ast \rho_s$, which we
shall call the \defword{heat kernel}; it is also shown that $\rho_s$
is $C^\infty$ and strictly positive.  Since $e^{s \Delta/4}$ is
Markovian, the heat kernel measure $\rho_s\,dm$ is a probability
measure.  

\begin{notation}
  For $s > 0$ and $0 < p < \infty$, we write $L^p(\rho_s)$ as short
  for $L^p(G, \rho_s\, dm)$.  As usual, for $0 < p < 1$, the vector
  space $L^p(\rho_s)$ is equipped with the topology induced by the complete
  translation-invariant metric $d(f,g) = \int|f-g|^p\,\rho_s\,dm$.
  Nonetheless $\|f\|_{L^p(\rho_s)}$ will still mean
  $\left(\int|f|^p\,\rho_s\,dm\right)^{1/p}$, even for the case $0 < p
  < 1$ in which it does not define a norm.
\end{notation}

Since $\rho_s$ is bounded, and bounded below on compact sets, any
sequence converging in $L^p(\rho_s)$ also converges in
$L^p_{\mathrm{loc}}(m)$.  As such, if $f_n$ are holomorphic functions
and $f_n \to f$ in $L^p(\rho_s)$, then we also have $f_n \to f$
uniformly on compact sets, and so $f$ is holomorphic.  Thus 
$L^p(\rho_s) \cap \mathcal{H}$ is closed in $L^p(\rho_s)$.

We record here some estimates for the heat kernel
\begin{theorem}
\label{heat-kernel-upper-lower}
  For each $0 < \epsilon < 1$ there are constants $C,C'$ 
   such that for every $x \in G$ and $s > 0$,
  \begin{equation}
    \frac{C}{m(B(e, \sqrt{s}))} e^{-d(e,x)^2/(1-\epsilon)s} \le \rho_s(x) \le \frac{C'}{m(B(e, \sqrt{s}))} e^{-d(e,x)^2/(1+\epsilon)s}
  \end{equation}
  where $m(B(e,\sqrt{s}))$ is the Lebesgue (Haar) measure of the
  $d$-ball centered at the origin (or any other point) of radius
  $\sqrt{s}$.
\end{theorem}
\begin{proof}
  The upper bound is Theorem IV.4.2 of \cite{purplebook}.  The lower
  bound is Theorem 1 of \cite{varopoulos-II}.  Note that our choice to
  consider the semigroup $e^{s \Delta/4}$ rather than $e^{s\Delta}$
  accounts for a missing factor of $4$ in the exponents compared to
  the results stated in \cite{purplebook, varopoulos-II}.
\end{proof}

\begin{theorem}
  \label{heat-kernel-deriv}
  Suppose $\xi_1, \dots, \xi_k \in \mathfrak{g}$.  Let $m$ be a nonnegative integer, $r \ge
  0$, and $0 < s < t < \infty$.  There is a constant $C$
  such that for all $y \in G$
  \begin{equation}\label{heat-kernel-deriv-eq}
    \sup_{d(x,e) < r} \left\vert \left(\frac{d^m}{ds^m} \widetilde{\xi}_1 \cdots
    \widetilde{\xi}_k \rho_s\right)(y \cdot x) \right\vert \le C \rho_t(y).
  \end{equation}
\end{theorem}

\begin{proof}
  This is a special case of Theorem IV.3.1 of \cite{purplebook}.  To
  reduce their statement to ours, note first that it suffices to
  assume the $\xi_i$ are all in $V_1$ (since, assuming H\"ormander's
  condition, any other left-invariant vector field may be written as a
  linear combination of commutators of vector fields from $V_1$).  We
  can also assume without loss of generality that the $\xi_i$ are
  orthonormal.  Then, in their notation, take $R = 1$, $\alpha = s$,
  $\beta = t$, and $\delta = r$.
\end{proof}

\begin{lemma}\label{hk-ratio-Lp}
  Let $s > 0$.
  \begin{enumerate}[(a)]
  \item \label{forall-t-exists-p} For every $t > s$ there exists $p >
    1$ such that $\rho_t / \rho_s \in L^p(\rho_s)$.  
  \item \label{forall-p-exists-t} For every $p \ge 1$ there exists $t > s$ such that $\rho_t /
    \rho_s \in L^p(\rho_s)$.
  \end{enumerate}
\end{lemma}

\begin{proof}
Let $\epsilon > 0$.
By Theorem \ref{heat-kernel-upper-lower}, for any $0 < s < t$, any $p
> 1$, and any
$\epsilon > 0$ we may find a constant $C(s,t,\epsilon)$ such that
\begin{align*}
  \left|\frac{\rho_t(x)}{\rho_s(x)}\right|^p \rho_s(x) &= \frac{\rho_t(x)^p}{\rho_s(x)^{p-1}}
   \\
  &\le C(s,t,\epsilon) \exp\left(-\left(\frac{p}{(1+\epsilon)t} -
  \frac{p-1}{(1-\epsilon)s}\right) d(e,x)^2\right)
\end{align*}
where the $m(B(e, \sqrt{\cdot}))$ factors have been absorbed into
$C(s,t,\epsilon)$.  Let $A = A(p,s,t,\epsilon) = \left(\frac{p}{(1+\epsilon)t} -
  \frac{p-1}{(1-\epsilon)s}\right)$ be the bracketed quantity in the
  exponent; if $A > 0$ then by (\ref{d-scaling}) the right side will
  be integrable with respect to $m$, implying the desired conclusion.

For (\ref{forall-t-exists-p}), suppose $t > s$ are given.  Fix any $\epsilon \in
(0,1)$.  As $p \downarrow 1$ we have $A \to \frac{1}{(1+\epsilon)t} >
0$, so for any $p$ sufficiently close to $1$ we get $A > 0$ and hence
$\rho_t / \rho_s \in L^p(\rho_s)$.

For (\ref{forall-p-exists-t}), suppose $s > 0$ and $p \ge 1$ are
given.  Without loss of generality we can assume $p > 1$ (since $L^1(\rho_s)
\supset L^p(\rho_s)$ for any $p > 1$).  Choose $t$ with $s < t < \frac{p}{p-1}
s$.  Then as $\epsilon \downarrow 0$ we have $A \to \frac{p}{t} -
\frac{p-1}{s} > 0$, so for any sufficiently small $\epsilon$ we get $A
> 0$.
\end{proof}

\begin{lemma}\label{log-drho}
  For any $\xi \in \mathfrak{g}$ and any $s>0$ we have $\widetilde{\xi} \log \rho_s
  \in \bigcap_{p \ge 1} L^p(\rho_s)$. 
\end{lemma}

\begin{proof}
  Fix $p \ge 1$.  By Lemma \ref{hk-ratio-Lp}(\ref{forall-p-exists-t})
  we can choose $t > s$ such that $\rho_t/\rho_s \in L^p(\rho_s)$.
  Then by Theorem \ref{heat-kernel-deriv}, taking any $r>0$ and $x=e$,
  there is a constant $C$ such that $\tilde{\xi} \rho_s \le C
  \rho_t$.  As such, by the chain rule we have
  \begin{equation*}
    \widetilde{\xi} \log \rho_s = \frac{\widetilde{\xi}
      \rho_s}{\rho_s} \le \frac{\rho_t}{\rho_s} \in L^p(\rho_s).
  \end{equation*}
\end{proof}

\begin{lemma}
  The heat kernel $\rho_s$ obeys the scaling relation
  \begin{equation}\label{rho-dilate}
    \rho_s(\delta_\lambda(y)) = |\lambda|^{-2D} \rho_{s|\lambda|^{-2}}(y).
  \end{equation}
\end{lemma}

\begin{proof}
  This follows from the corresponding scaling properties of the
  semigroup $e^{s\Delta/4}$ (\ref{semigroup-dilate}) and of the Haar
  measure $m$ (\ref{m-dilate}).
\end{proof}

\section{Dirichlet forms and operators}\label{sec:forms}

For the rest of the paper, fix some $a > 0$.  Henceforward $L^p$ by
  itself will, unless otherwise specified, refer to $L^p(\rho_a)$. 

\begin{notation}\label{Q-A-notation}
  Let $Q_0$ be the positive quadratic form on $L^2(\rho_a)$ defined on the
  domain $C_c^\infty(G)$ by
  \begin{equation}
    Q_0(f_1, f_2) = \int_G h(df_1, d\bar{f}_2)\, \rho_a\,dz = \int_G g(\nabla f_1,
    \nabla \bar{f}_2)\, \rho_a\,dz
  \end{equation}
  and let $Q$ be its closure, with domain $\mathcal{D}(Q)$, so that
  $(Q, \mathcal{D}(Q))$ is a Dirichlet form on $L^2(\rho_a)$.  Note
  that $\mathcal{D}(Q)$ is a Hilbert space under the \defword{energy norm}
  $(f,g)_Q = (f,g)_{L^2(\rho_a)} + Q(f,g)$.  Let
  $(A, \mathcal{D}(A))$ be the generator of $Q$, i.e., $A$ is the
  unique self-adjoint operator on $L^2(\rho_a)$ having domain
  $\mathcal{D}(A) \subset \mathcal{D}(Q)$ and satisfying $\int_G (Af_1)
  \bar{f}_2 \,\rho_a\,dz = Q(f_1,f_2)$ for all $f_1 \in \mathcal{D}(A)$, $f_2
  \in \mathcal{D}(Q)$.  
\end{notation}

On smooth functions $f \in \mathcal{D}(A) \cap C^\infty(G)$, integration by parts gives
\begin{equation}\label{A-smooth}
  A f = d^* d f = -\Delta f - g(\nabla f, \nabla \log \rho_a) =
  -\Delta f - h(df, d \log \rho_a).
\end{equation}    

{\nate The operator $A = d^* d$ can be seen as an analogue of the
Ornstein--Uhlenbeck operator in this noncommutative Lie group
setting.  Such operators have attracted substantial interest in the
literature, including the study of functional inequalities such as
Poincar\'e inequalities.  Papers which study these operators (in the
setting of real Lie groups) include
\cite{baudoin-hairer-teichmann, lust-piquard-ornstein-uhlenbeck,
  melcher08}.}

\begin{remark}\label{euclidean-remark}
  When $\mathfrak{g}$ is abelian (i.e. the Lie bracket is 0) then $G$
  is Euclidean space $\mathbb{C}^n$ (with its usual additive group
  structure).  If we take $h$ to be the usual positive definite
  Euclidean inner product, then everything reduces to the Euclidean
  case: $\nabla$ and $\Delta$ are the usual gradient and Laplacian,
  $d$ is Euclidean distance, $\rho_s$ is the Gaussian heat kernel
  $\rho_s(z) = (\pi s)^{-n} e^{-|z|^2/s}$, and
  $A$ is the Ornstein--Uhlenbeck operator.
\end{remark}

\begin{definition}\label{holomorphic-def}
  We will say that $A$ is a \defword{holomorphic} operator if it maps
  holomorphic functions to holomorphic functions;
  i.e. $A(\mathcal{D}(A) \cap \mathcal{H}) \subset \mathcal{H}$.
\end{definition}

In our setting, the operator $A$ is \emph{not} holomorphic (except in
the abelian case $G = \mathbb{C}^n$); see Theorem
\ref{A-not-holomorphic} below.  So our setting
stands in contrast to that of \cite{gross-hypercontractivity-complex},
in which most of the main results were proved under the hypothesis
that the operator $A$ should be holomorphic.  

Since the phenomenon of strong hypercontractivity is quite specific to
the holomorphic category, it is not reasonable to expect it to hold
for an operator that does not preserve holomorphicity.
As such, our
main object of study will not be $A$ itself, but rather the operator
$B$ defined as follows.

\begin{notation}
  The restriction $Q|_{\mathcal{H}}$ of $Q$ to the domain
  $\mathcal{D}(Q) \cap \mathcal{H}$ is a positive closed quadratic
  form on the Hilbert space $\mathcal{H} \cap L^2(\rho_a)$.  Let $(B,
  \mathcal{D}(B))$ be its generator, so that $B$ is a self-adjoint
  operator on $\mathcal{H} \cap L^2(\rho_a)$.
\end{notation}

We intend to think of $B$ as the ``holomorphic projection'' of the
operator $A$.  In Section \ref{sec:poly}, we shall discuss the
precise sense in which this is true.  For now, let us observe that
\begin{equation}\label{DAcapH-subset-DB}
  \mathcal{D}(A) \cap \mathcal{H} \subset \mathcal{D}(B).
\end{equation}
To see this, note that for $f \in \mathcal{D}(A) \cap \mathcal{H}
\subset \mathcal{D}(Q) \cap \mathcal{H}$ and
$g \in \mathcal{D}(Q) \cap \mathcal{H}$, we have $|Q(f,g)| =
\left|(Af, g)_{L^2}\right| \le \|Af\|_{L^2} \|g\|_{L^2}$, and so $f$
is in the domain of the generator of $Q|_{\mathcal{H}}$, namely $B$.

\section{Density properties of holomorphic polynomials}\label{sec:poly}

\begin{notation}{$\mathcal{H}$ will denote the set of holomorphic functions on $G$.
}
\end{notation}

\begin{theorem}\label{poly-density}
 \mbox{}
\begin{enumerate}[(a)]
\item $\mathcal{P}$ is dense in $\mathcal{H}\cap L^{p}(\rho_{a})$ for
$1\leq p<\infty$. \label{P-dense}

\item $\mathcal{P}\subset\mathcal{D}(Q)$ and is a core for
  $Q|_{\mathcal{H}}$. \label{P-core-Q}  In particular, from
  (\ref{P-dense}), $Q|_{\mathcal{H}}$ is densely defined in
  $\mathcal{H} \cap L^2(\rho_a)$.

\item If $j\neq k$ then $\mathcal{P}_{j}\perp\mathcal{P}_{k}$ in both
$L^{2}(\rho_{a})$ and in energy norm. \label{P-perp}

\item $\mathcal{H}\cap L^{2}(\rho_{a})=\bigoplus_{k=0}^{\infty}%
\mathcal{P}_{k}$. \label{L2-P-decomp}

\item $\mathcal{H}\cap\mathcal{D}(Q)=\bigoplus_{k=0}^{\infty}\mathcal{P}_{k}$
(convergence in energy norm) \label{Q-P-decomp}

\item $\mathcal{P}\subset\mathcal{D}(B)$ and is a core for $B$. \label{P-core-B}
\end{enumerate}
\end{theorem}

{\nate
  \begin{remark}
  It is interesting to contrast Theorem \ref{poly-density} with
  \cite[Proposition 8]{lust-piquard-ornstein-uhlenbeck} (credited to
  W.~Hebisch), in which it is shown that the result is typically false if we
  drop the word ``holomorphic''.  Specifically, when $G$ is a
  (real) stratified Lie group, the (not necessarily holomorphic)
  polynomials are dense in $L^2(\rho_a)$ if and only if $G$ has step
  at most 4.  
  \end{remark}
}

\begin{proof}
The proofs are slight variants of the proof of \cite[Lemma
  5.4]{gross-hypercontractivity-complex}.

For (\ref{P-dense}), to begin, it follows from the upper bound in Theorem
\ref{heat-kernel-upper-lower}, using polar coordinates and the homogeneity
of $d$, that $\mathcal{P} \subset  L^p(\rho_a)$.

Let
\begin{align}
F_{n}(\theta) &  =\frac{1}{2\pi n}\ \sum_{k=0}^{n-1}\ \sum_{j=-k}%
^{k}e^{ij\theta}\nonumber\\
&  =\frac{1}{2\pi n}\ \frac{\sin^{2}(n\theta/2)}{\sin^{2}(\theta
/2)}\label{e.d2.16}%
\end{align}
denote Fejer's kernel \cite[\S{}13.31]{titchmarsh-book}. We observe that
\begin{align}
  &\int_{-\pi}^{\pi} F_n(\theta)\,d\theta = 1 \\
  &\int_{-\pi}^{\pi} F_n(\theta) e^{i \ell \theta}\,d\theta = 0, &&
  \ell \ge n  \label{Fn-zero} \\
  &\lim_{n \to \infty} \int_{-\pi}^{\pi} F_n(\theta)
  \varphi(\theta)\,d\theta = \varphi(0), && \varphi \in C([-\pi, \pi]). 
\end{align}
Define 
$$V_\theta f := f \circ \delta_{e^{i\theta}}$$
for any function $f$ on $G$. If $f\in\mathcal{H}$ and is written $f =
\sum_{k=0}^\infty f_k$ as in (\ref{e.d2.14}), with $f_k \in \mathcal{P}_k$, then
\begin{equation}\label{Vtheta-decomp}
(V_{\theta}f)(z)=\sum_{k=0}^{\infty}e^{ik\theta}f_{k}(z).
\end{equation}
The convergence is uniform on $\theta\in\lbrack-\pi,\pi]$ for each $z\in G$
because the function $\theta\mapsto f(\delta_{e^{i\theta}}z)$ is smooth and
periodic with period $2\pi$. Now let
\begin{equation}
g_{n}(z):=\int_{-\pi}^{\pi}F_{n}(\theta)(V_{\theta}f)(z)\,d\theta. \label{e.d2.17}%
\end{equation}
Using (\ref{Vtheta-decomp}), Fubini's theorem, and (\ref{Fn-zero}), we
see that $g_n$  is a linear combination of
$f_{0},f_{1},\ldots,f_{n-1}$ and is therefore in $\mathcal{P}$. (We
can justify the application of Fubini's theorem using the fact that
$\sum_{k=0}^\infty f_k(z)$ is the Taylor series for $u(\lambda)$, as
defined in (\ref{u-def}), at $\lambda = 1$, and therefore converges
absolutely.)  Since the map $\delta_{e^{i\theta}}:G\rightarrow G$
preserves the measure $\rho_{a}(x)dx$ (see (\ref{m-dilate}),
(\ref{rho-dilate})), the operators $V_{\theta}$ are isometries in
$L^{p}(G,\rho_{a}(x)dx)$ for $0<p<\infty$.  Moreover, the map
$\theta\mapsto V_{\theta}$ is strongly continuous in $L^{p}(\rho_a)$
for $1\leq p<\infty$: for bounded continuous $f : G \to \mathbb{R}$,
dominated convergence gives $V_\theta f \to f$ in $L^p(\rho_a)$ as
$\theta \to 0$, and the case of general $f \in L^p(\rho_a)$ follows by
density.

Thus if $1\leq p<\infty$ and $f\in\mathcal{H}\cap
L^{p}(\rho_{a})$ then we have
\begin{equation}\label{Vtheta-computation}
\begin{split}
\Vert f-g_{n}\Vert_{L^{p}} &  =\Big\|\int_{-\pi}^{\pi}F_{n}(\theta
)(f-V_{\theta}f)\,d\theta\Big\|_{L^{p}}\\
&  \leq\int_{-\pi}^{\pi}F_{n}(\theta)\Vert f-V_{\theta}f\Vert_{L^{p}}\,d\theta\\
&  \to 0\quad\text{as}\quad n\rightarrow\infty
\end{split}
\end{equation}
by Minkowski's inequality for integrals.
This proves part (\ref{P-dense}).

To prove part (\ref{P-core-Q}), recall that by Lemma \ref{dilate-poly},
if $f\in\mathcal{P}_{k}$ and $\xi\in V_{1}$ then
$\widetilde{\xi}f\in\mathcal{P}_{k-1} \subset L^2(\rho_a)$.  Hence
$|\nabla f|^{2}$ is in $L^{1}(\rho_{a})$.  Moreover, multiplying $f$
by a sequence $\varphi_{n}$ of cutoff functions in $C_{c}^{\infty}(G)$ which
converge to 1 boundedly and such that
$\widetilde{\xi}\varphi_{n} \to 0$ boundedly, one sees that
$f\in\mathcal{D}(Q)$. So $\mathcal{P}\subset\mathcal{D}(Q)$. By
(\ref{gZ}) and (\ref{Z-dilation}), for any smooth $f$ we have
\begin{equation}
|\nabla(f\circ\delta_{e^{i\theta}})|^{2}(z)=|\nabla f|^{2}(\delta_{e^{i\theta
}}z).
\end{equation}
Since $\rho_{a}(x)dx$ is preserved by the map $\delta_{e^{i\theta}}$ it
follows that
\[
Q(V_{\theta}f)=Q(f)\quad\text{for all}\quad f\in\mathcal{D}(Q)
\]
and in particular for all $f\in\mathcal{H}\cap\mathcal{D}(Q)$. So $V_{\theta}$
is unitary on $\mathcal{H}\cap\mathcal{D}(Q)$ in the energy norm, $[\Vert
f\Vert_{L^{2}}^{2}+Q(f)]^{1/2}$.  Now if
$f\in\mathcal{H}\cap\mathcal{D}(Q)$ and we define the polynomials
$g_{n}$ as in (\ref{e.d2.17}), we can differentiate under the integral
sign to see that 
\begin{equation}
\widetilde{\xi} g_n(z) =
\int_{-\pi}^{\pi}F_{n}(\theta)(\widetilde{\xi} V_{\theta}
f)(z)\,d\theta = \int_{-\pi}^{\pi}F_{n}(\theta)e^{i\theta} (V_{\theta} \widetilde{\xi} 
f)(z)\,d\theta.
\end{equation}  Then, similar to \ref{Vtheta-computation}, we have
\begin{equation}
\begin{split}
  \left\Vert \widetilde{\xi}f - \widetilde{\xi}g_{n}\right\Vert_{L^{2}} 
  &=\left\Vert \int_{-\pi}^{\pi} F_{n}(\theta)
  (\widetilde{\xi} f-e^{i\theta}V_{\theta}\widetilde{\xi} f)
  \,d\theta \right\Vert_{L^{p}}\\ 
  &\leq \int_{-\pi}^{\pi} F_{n}(\theta) 
  \left\Vert \widetilde{\xi} f-e^{i\theta} V_{\theta}\widetilde{\xi}
  f\right\Vert_{L^{p}}\,d\theta \\ 
  &\to 0 \quad\text{as $n\to\infty$}.
 \end{split}
 \end{equation}
 It follows that $g_n \to f$ in energy norm.  Hence
 $\mathcal{P}$ is a core for $Q\mid\mathcal{H}$.

 Now if $f\in\mathcal{P}_{n}$ and $g\in\mathcal{P}_{k}$ then $(V_{\theta
 }f)(z)=e^{in\theta}f(z)$ and $(V_{\theta}g)(z)=e^{ik\theta}g(z)$ by
 (\ref{e.d2.10}). Hence $(f,g)_{L^{2}}=(V_{\theta}f,V_{\theta}g)_{L^{2}%
 }=e^{i(n-k)\theta}(f,g)_{L^{2}}$ for all real $\theta$. So if $n\neq k$ then
 $(f,g)_{L^{2}}=0$. Moreover, $\widetilde{\xi}f\in\mathcal{P}_{n-1}$ and
 $\widetilde{\xi}g\in\mathcal{P}_{k-1}$ if $\xi\in V_{1}$. So if $n\neq k$ then
 $Q(f,g)=0$. This proves part (\ref{P-perp}). Parts (\ref{L2-P-decomp})
 and (\ref{Q-P-decomp}) now follow from parts (\ref{P-dense}),
 (\ref{P-core-Q}), (\ref{P-perp}).

 To prove part (\ref{P-core-B}), assume first that $g\in\mathcal{P}_{n}$. Let $f\in
 \mathcal{H}\cap\mathcal{D}(Q)$. By part (\ref{Q-P-decomp}) we may write $f=\sum_{k=0}^{\infty}f_{k}$ with
 $f_{k}\in\mathcal{P}_{k}$, by part (e), which also yields
 \[
 |Q(g,f)|=|Q(g,f_{n})|\leq Q(g)^{1/2}Q(f_{n})^{1/2}.
 \]
 Since $\mathcal{P}_{n}$ is finite dimensional (Corollary
 \ref{Pk-finite-dim}) there is a constant $C_{n}$ such that
 $Q(f_{n})\leq C_{n}^{2}\Vert f_{n}\Vert_{L^{2}}^2$. Since the
 functions $f_{k}$ are orthogonal in the $L^{2}$ inner product we have
 $\Vert f_{n} \Vert_{L^{2}}^{2}\leq\Vert f\Vert_{L^{2}}^{2}$. Thus
 $|Q(g,f)|\leq Q(g)^{1/2}C_{n}\Vert f\Vert_{L^{2}}$. Hence
 $g\in\mathcal{D}(B)$ and we have shown
 $\mathcal{P}\subset\mathcal{D}(B)$.

 Now suppose that $h\in\mathcal{D}(B)$. Define $h_{n}(z)=\int_{-\pi}^{\pi}%
 F_{n}(\theta)(V_{\theta}h)(z)\,d\theta$. As we have seen, $h_{n}\in\mathcal{P}$.
 We will show that $h_{n}\rightarrow h$ in the graph norm of $B$, using the
 fact that $V_{\theta}$ is unitary in both of the Hilbert spaces $\mathcal{H}%
 L^{2}$ and $\mathcal{H}\cap\mathcal{D}(Q)$. If $g\in\mathcal{H}\cap
 \mathcal{D}(Q)$ then
 \begin{equation}
 (V_{\theta}Bh,g)=(Bh,V_{-\theta}g)=Q(h,V_{-\theta}g)=Q(V_{\theta}h,g).
 \label{e.d2.18}%
 \end{equation}
 Since the left side is continuous in $g$ in the $L^{2}$ norm so is
 $Q(V_{\theta}h,g)$. Hence $V_{\theta}h\in\mathcal{D}(B)$ and
 \begin{equation}
 V_{\theta}Bh=BV_{\theta}h, \qquad h\in\mathcal{D}(B). \label{e.d2.19}%
 \end{equation}
 Although this equality is of interest in itself we will actually use
 (\ref{e.d2.18}) a little differently. Multiply equation (\ref{e.d2.18}) by
 $F_{n}(\theta)$ and integrate over $[-\pi,\pi]$. The integral can be taken
 inside both the $L^{2}$ and energy inner products because $V_{\theta}$ is
 strongly continuous in both spaces. We obtain
 \[
 \left(\int_{-\pi}^{\pi}F_{n}(\theta)V_{\theta}Bh\,d\theta,g\right)=Q(h_{n}%
 ,g)\quad\forall\ g\in\mathcal{H}\cap\mathcal{D}(Q).
 \]
 So
 \[
 \int_{-\pi}^{\pi}F_{n}(\theta)V_{\theta}Bh\,d\theta=Bh_{n}.
 \]
 As $n\rightarrow\infty$ the left side converges to $Bh$ in $L^{2}$ norm. Thus
 $h_{n}\rightarrow h$ and $Bh_{n}\rightarrow Bh$. Hence $\mathcal{P}$ is a core
 for $B$.
 \end{proof}

Let us remark on the requirement $p \ge 1$ in Theorem
\ref{poly-density}(\ref{P-dense}).  Our proof fails for $0 < p < 1$
because the inequality in (\ref{Vtheta-computation}) would go the
wrong way.  

However, in the Euclidean case $G = \mathbb{C}^n$ (see Remark
\ref{euclidean-remark}), where $\rho_a$ is the Gaussian heat kernel,
it is known that in fact $\mathcal{P}$ is dense in $L^p(\rho_a)$ for
$0 < p < 1$.  This is a consequence of a theorem of Wallst\'en
\cite[Theorem 3.1]{wallsten-hankel}, from which it follows that the set
$\mathcal{E}$ of holomorphic functions of the form $f(z) =
\sum_{j=1}^m a_j e^{z_j \cdot \bar{w}_j}$, with $a_j \in \mathbb{C}$
and $w_j \in \mathbb{C}^n$, is dense in $L^p(\rho_a)$.  Since
$\mathcal{E} \subset L^1$, we have that $L^1$ is dense in $L^p$.  But
since $\mathcal{P}$ is dense in $L^1$ and the inclusion $L^1 \subset
L^p$ is continuous, we have $\mathcal{P}$ dense in $L^p$ as well.
Unfortunately for us, Wallst\'en's argument relies heavily on the
simple structure of the Gaussian, and it is not clear whether it can
adapted to a general complex Lie group with a H\"ormander metric $h$.

\begin{question}
  For general $(G,h)$, is $\mathcal{P}$ dense in $L^p(\rho_a)$ for $0
  < p < 1$?
\end{question}
  
In light of this issue, we adopt the following function spaces on
which to prove our main results.

\begin{notation}\label{HLp-def}
  For $1 \le p < \infty$, let $\mathcal{H} L^p(\rho_a) = \mathcal{H}
  \cap L^p(\rho_a)$.  For $0 < p < 1$, let $\mathcal{H} L^p(\rho_a)$
  be the $L^p$-closure of $\mathcal{H} \cap L^2(\rho_a)$, which may or
  may not equal $\mathcal{H} \cap L^p(\rho_a)$.
\end{notation}

In particular, by this definition, $\mathcal{P}$ is dense in
$\mathcal{H} L^p(\rho_a)$ for every $0 < p < \infty$.  Also, for $0 <
p < q < \infty$, $\mathcal{H} L^q$ is dense in $\mathcal{H} L^p$.

\begin{remark}\label{HLp-remark}
  Our spaces $\mathcal{H} L^p$ are defined differently from the spaces
  $\mathcal{H}^p$ used in \cite{gross-hypercontractivity-complex}, but
  in our current setting they are equal.  
  \begin{itemize}
  \item For $p=2$, 
    \cite{gross-hypercontractivity-complex} defines $\mathcal{H}^2$ as
    the $L^2$-closure of $\mathcal{H} \cap \mathcal{D}(Q)$; for us,
    Theorem \ref{poly-density}(\ref{P-dense},\ref{P-core-Q}) shows this
    equals $\mathcal{H} \cap L^2$.  
  \item For $p > 2$,
    \cite{gross-hypercontractivity-complex} defines $\mathcal{H}^p$ as
    $\mathcal{H}^2 \cap L^p$; for us this equals $\mathcal{H} \cap L^2
    \cap L^p = \mathcal{H} \cap L^p$.  
  \item For $0 < p < 2$, \cite{gross-hypercontractivity-complex}
    defines $\mathcal{H}^p$ as the $L^p$ closure of $\mathcal{H}^2$.
    For $0 < p < 1$ this is precisely our definition; for $1 \le p <
    2$, this equals $\mathcal{H} \cap L^p$ since $\mathcal{H} L^2$ is
    dense in $\mathcal{H} L^p$.
  \end{itemize}
  In the cases considered by \cite{gross-hypercontractivity-complex},
  it was possible that $\mathcal{H}^p$ was very different from
  $\mathcal{H} \cap L^p$; see the counterexamples in \cite[Section 5]{gross-hypercontractivity-complex}.
\end{remark}

We now return to the question of in what sense $B$ is a ``holomorphic
projection'' of $A$.  Let $P_{\mathcal{H}}$ be orthogonal projection
from $L^2$ onto the closed subspace $\mathcal{H} L^2$.

\begin{proposition}\label{B-holomorphic-projection}
  For $f \in \mathcal{D}(A) \cap \mathcal{D}(B)$, we have $Bf =
  P_{\mathcal{H}} Af$.
\end{proposition}

\begin{proof}
  For any $g \in \mathcal{H} \cap \mathcal{D}(Q)$, we have
  \begin{align*}
    (Bf, g)_{L^2} = Q(f,g) = (Af, g)_{L^2} = (P_{\mathcal{H}} A f, g)_{L^2}.
  \end{align*}
  Since $\mathcal{H} \cap \mathcal{D}(Q)$ is dense in $\mathcal{H}
  \cap L^2$ we must have $Bf = P_{\mathcal{H}} Af$.
\end{proof}

To make the previous proposition more interesting, we should show that
$\mathcal{D}(A) \cap \mathcal{D}(B)$ is reasonably large.

\begin{proposition}\label{P-DA}
  $\mathcal{P} \subset \mathcal{D}(A)$.
\end{proposition}

\begin{proof}
  Let $f \in \mathcal{P}$, and let $\varphi = - \Delta f - h(df,
  d\log\rho_a)$ be the function which, as in (\ref{A-smooth}), ought
  to equal $Af$.  Integration by parts shows that for any $\psi \in
  C^\infty_c(G)$ we have $Q(f, \psi) = \int_G \varphi \bar{\psi}
  \rho_a\,dm$, so if we can show $\varphi \in L^2(\rho_a)$, we will
  have $|Q(f,\psi)| \le \|\varphi\|_{L^2} \|\psi\|_{L^2}$, implying
  that $f \in \mathcal{D}(A)$ and moreover $Af = \varphi$.

  Since $f$ is holomorphic, $\Delta f =0$ so we have
  \begin{equation}
    \varphi = -h(d f, d \log \rho_a) = -  \sum_j Z_j f \bar{Z}_j \log
    \rho_a
  \end{equation}
  using (\ref{gZ}) and $\bar{Z}_j f = 0$.  By Lemma \ref{dilate-poly},
  $Z_j f \in \mathcal{P} \subset \bigcap_{q \ge 1} L^q(\rho_a)$, and
  by Lemma \ref{log-drho}, $\bar{Z}_j \log \rho_a \in \bigcap_{p \ge 1}
  L^p(\rho_a)$, so by H\"older's inequality, $\varphi \in
  L^2(\rho_a)$ as desired.
\end{proof}

(A similar argument would show that any $L^2$ holomorphic function with its
first derivatives in $L^{2+\epsilon}$ is also in $\mathcal{D}(A)$.)

In particular we have $\mathcal{P} \subset \mathcal{D}(A) \cap
\mathcal{D}(B)$, so $Bf = P_{\mathcal{H}} Af$ for all polynomials.  

In the case that $A$ is holomorphic, we actually have that $B$ is
simply the restriction of $A$ to $\mathcal{D}(A) \cap \mathcal{H}$.
We already showed in (\ref{DAcapH-subset-DB}) that $\mathcal{D}(A)
\cap \mathcal{H} \subset \mathcal{D}(B)$.  For the other direction,
let $f \in \mathcal{D}(B)$; by Theorem
\ref{poly-density}(\ref{P-core-B}) we can find a sequence $p_n \in
\mathcal{P}$ with $p_n \to f$ and $Bp_n \to Bf$ in $L^2$.  But $B p_n
= P_{\mathcal{H}} A p_n = A p_n$ if $A$ is holomorphic, so $A p_n$
converges, and since $A$ is closed we have $f \in \mathcal{D}(A)$ and
$Af = Bf$.

It is conceivable that even when $A$ is not holomorphic, we might
get $\mathcal{D}(B) = \mathcal{D}(A) \cap \mathcal{H}$, in which case
$B$ is simply the restriction of $P_{\mathcal{H}} A$ to
$\mathcal{D}(A) \cap {H}$, i.e. the literal holomorphic projection of
$A$.  However, we do not have a proof of this.

\begin{question}
  Under what conditions does $\mathcal{D}(B) = \mathcal{D}(A) \cap \mathcal{H}$?
\end{question}

\section{Dilations and the operator $B$}\label{sec:B}

In this subsection, we show that in fact the operator $B$ is just a
constant multiple of the vector field $Z$ introduced in
(\ref{e.d2.8}): $B = \frac{2}{a} Z$.  Along the way, we establish some
lemmas that will also be useful in future computations.

\begin{remark}
  To see that $B = \frac{2}{a} Z$ is a plausible statement, consider
  the Euclidean case $G = \mathbb{C}^n$ as in Remark
  \ref{euclidean-remark}.  Here $A$ is the Ornstein--Uhlenbeck
  operator  $Af = -\Delta f + \frac{1}{a} z\cdot \nabla f$;  
   since this
  is a holomorphic operator, $B$ is simply the restriction of $A$ to
  holomorphic functions.  For holomorphic $f$ we have $\Delta f = 0$
  and  $z\cdot \nabla f 
   = 2 \sum_{j=1}^n z_j \frac{\partial f}{\partial z_j}$.  On the
  other hand, as in (\ref{e.d2.9}), in this case we have $Zf =
  \sum_{j=1}^n z_j \frac{\partial f}{\partial z_j}$ (note that all the
  $c_j$ are 1).
\end{remark}

\begin{notation}\label{poly-growth}
  Let us introduce a class of convenient functions with which to work.
  We will say a function $f : G \to \mathbb{C}$ has
  \defword{polynomial growth} if there are constants $C,N$ such that
  $|f(z)| \le C (1 + d(e,z))^N$ for all $z$.  Then we let $C^2_p(G)$
  denote the class of all $f \in C^2(G)$ such that $f, \xi_j f, \xi_j
  \xi_k f, Xf, Yf$ all have polynomial growth.
\end{notation}

It is immediate that $\mathcal{P} \subset C^2_p(G)$, and if $f,g$ are
in $C^2_p(g)$ then so are $f \circ \delta_\lambda$, $\bar{f}$, $f+g$,
and $fg$.  Moreover, if $u : \mathbb{C} \to \mathbb{C}$ is a $C^2$
function with bounded first and second derivatives, then $u(f)$ is
also in $C^2_p$.  This is certainly not the broadest class of
functions for which the results below will hold, but it is sufficient
for our purposes and simplifies several of the arguments.

\begin{lemma}\label{ds-Delta-X}
  If $f \in C^2_p(G)$, then $s \mapsto \int_G f \,\rho_s\,dm$ is
  differentiable and
  \begin{equation}\label{ds-Delta-X-eq}
    \frac{d}{ds} \int_G f \,\rho_s\,dm = \frac{1}{4} \int_G \Delta f\, \rho_s\, dm
    = \frac{1}{2s} \int_G Xf\,\rho_s\,dm.
  \end{equation}
\end{lemma}

\begin{proof}
  Suppose first that $f \in C^\infty_c(G)$.  Let $a(s) = \int_G
  f\,\rho_s\,dm$.  For the first equality, differentiating under the
  integral sign and then integrating by parts gives
  \begin{equation*}
    a'(s) = \int_G f\,\frac{d}{ds}\rho_s\,dm = \frac{1}{4} \int_G
    f\,\Delta \rho_s\,dm = \frac{1}{4} \int_G \Delta f \,\rho_s\,dm.
  \end{equation*}
  For the second equality, we use (\ref{semigroup-dilate}) to
  observe
  \begin{align*}
    \int_G (f \circ \delta_{e^r})\,\rho_s\,dm &= e^{s \Delta/4}(f
    \circ \delta_{e^r})(e) \\&
    = (e^{s e^{2r} \Delta/4}
    f)(\delta_{e^r}(e)) \\ 
    &=  (e^{s e^{2r} \Delta/4}
    f)(e) \\
    &= \int_G f \,\rho_{s e^{2r}}\,dm \\
    &= a(s e^{2r}).
  \end{align*}
  Now differentiating under
  the integral sign with respect to $r$ and then setting $r=0$, we get
  \begin{equation*}
    \int_G Xf\,\rho_s\,dm = \frac{d}{dr}\Big|_{r=0} a(s e^{2r}) = 2s\, a'(s)
  \end{equation*}
  which establishes the second equality of (\ref{ds-Delta-X-eq}).

  For the case of general $f \in C^2_p(G)$, let $\psi \in
  C^\infty_c(G)$ be a cutoff function which equals 1 on a neighborhood
  of $e \in G$, and set $\psi_n(x) = \psi(\delta_{1/n}(x))$.  Then
  $\psi_n \to 1$ boundedly.  It follows from (\ref{X-dilation})
  that $\widetilde{\xi}_j \psi_n \to 0$ and $\widetilde{\xi}_j
  \widetilde{\xi}_k \psi_n \to 0$ boundedly, at least for $\xi \in V_1$, and
  the same for general $\xi \in \mathfrak{g}$ by taking commutators.
  Then since $X,Y$ commute with $\delta_{1/n}$, we also have $X \psi_n
  \to 0$, $Y\psi_n \to 0$ boundedly.  Hence setting $f_n = \psi_n f$,
  we have constructed $f_n \in C^2_c(G)$ such that, pointwise,
  \begin{equation*}
    f_n \to f, \qquad \Delta f_n \to \Delta f, \qquad X f_n \to Xf,
  \end{equation*}
  and moreover, such that $f_n$ and its derivatives
  are controlled by $f$ and its derivatives.  In particular, there
  exist $C, N$ such that for all $n,x$ we have 
  \begin{equation*}
    |f_n(x)| + |\Delta f_n(x)| + |X f_n(x)| \le C(1+d(e,x))^N.
  \end{equation*}
  Now by integrating (\ref{ds-Delta-X-eq}), we have
  \begin{equation}\label{integrated}
    \int_G f_n\, (\rho_t - \rho_s)\,dm = \frac{1}{4}\int_s^t \int_G \Delta
    f_n\,\rho_{\sigma}\,dm\,d\sigma = \int_s^t \frac{1}{2\sigma} \int_G Xf_n\,\rho_\sigma\,dm\,d\sigma.
  \end{equation}
  By the Gaussian heat kernel upper bounds of Theorem
  \ref{heat-kernel-upper-lower}, we have
  \begin{equation*}
    \int_G C (1+d(e,x))^N \sup_{\sigma \in [s,t]} \rho_\sigma(x)
    \,m(dx) < \infty
  \end{equation*}
  and so by Fubini's theorem and dominated convergence, we can pass to
  the limit in (\ref{integrated}) as $n \to \infty$ to get 
  
  \begin{equation}
    \int_G f\, (\rho_t - \rho_s)\,dm =\frac{1}{4} \int_s^t \int_G \Delta
    f\,\rho_{\sigma}\,dm\,d\sigma = \int_s^t \frac{1}{2\sigma} \int_G Xf\,\rho_\sigma\,dm\,d\sigma.
  \end{equation}
  { Since the two integrals over $G$ are each continuous functions of
  $\sigma$,} then by the fundamental theorem of calculus, this is
  equivalent to the desired result.
\end{proof}

\begin{lemma}\label{int-Y-0}
  For $f \in C^2_p(G)$, we have $\int_G Yf\,\rho_s\,dm = 0$.
\end{lemma}

\begin{proof}
  This is similar to the previous proof.  By (\ref{semigroup-dilate})
  we have
  \begin{align*}
    \int_G (f \circ \delta_{e^{i\theta}})\,\rho_s\,dm &= e^{s \Delta/4}(f
    \circ \delta_{e^{i\theta}})(e) \\&
    = (e^{s |e^{i\theta}|^2 \Delta/4}
    f)(\delta_{e^{i \theta}}(e)) \\ 
    &=  (e^{s \Delta/4}
    f)(e) \\
    &= \int_G f \,\rho_{s}\,dm. 
  \end{align*}
  If $f \in C^2_c(G)$ we can differentiate under the integral sign
  with respect to $\theta$ and set $\theta = 0$ to get $\int_G
  Yf\,\rho_s\,dm = 0$.  For $f \in C^2_p(G)$, use cutoff functions.
\end{proof}

\begin{corollary}
\label{Z-symmetric} Suppose that $f, g \in \mathcal{P}$. Then
\begin{equation}
(Zf,g)_{L^{2}(\rho_{a})}=(f,Zg)_{L^{2}(\rho_{a})}.
\end{equation}
\end{corollary}

\begin{proof}
$-iY(f\bar{g})=(Z-\bar{Z})(f\bar{g})=(Zf)\bar{g}-f\overline{Zg}$. Since $f \bar{g}
\in C^2_p(G)$, by Lemma \ref{int-Y-0} the
integral with respect to $\rho_{a}\,dm$ is zero.
\end{proof}

\begin{theorem}
\label{ZB} Let $a>0$.  We have
\begin{equation}\label{ZB-domain}
  \mathcal{D}(B) = \{f \in \mathcal{H} L^2(\rho_a) : Zf \in
  L^2(\rho_a)\}
\end{equation}
and
\begin{equation}\label{ZB-formula}
Bf=\frac{2}{a}Zf\quad\text{for all}\quad f\in\mathcal{D}(B).
\end{equation}
\end{theorem}

\begin{proof}
  We begin by showing that (\ref{ZB-formula}) holds for $f \in
  \mathcal{P}$.  Suppose that $f$ and $g$ are in $\mathcal{P}$, and let $Z_j$ be as
defined in (\ref{Zj-def}).  First observe that
\begin{equation*}
  Z_j \bar{Z_j}(f \bar{g}) = Z_j \bar{Z}_j f \cdot \bar{g} + \bar{Z}_j
  f \cdot Z_j \bar{g} + Z_j f \cdot \bar{Z}_j \bar{g} + f \cdot Z_j
  \bar{Z}_j \bar{g} = Z_j f \cdot \overline{Z_j g}.
\end{equation*}
The first, second and fourth terms of the middle expression vanish
because $\bar{Z}_j f = 0$ and $Z_j \bar{Z}_j \bar{g} = \bar{Z}_j Z_j
\bar{g} = 0$ (since $Z_j$ is of type (1,0) and commutes with
$\bar{Z}_j$).  So by (\ref{gZ}) and (\ref{sub-Laplacian-def}) we have
\begin{equation*}
  h(df, d\bar{g}) = \frac{1}{2} \Delta(f \bar{g}).
\end{equation*}
Note that $f \bar{g} \in C^2_p(G)$.  Thus multiplying by $\rho_a$ and integrating, we have
\begin{align*}
(Bf,g)_{L^2(\rho_a)}  &  =Q(f,g)\\
&  =\frac{1}{2} \int_{G}\Delta(f\overline{g})\rho_{a}\,dm\\
&  =\frac{1}{a}\int_{G}X(f\overline{g})\rho_{a}\,dm && \text{by Lemma \ref{ds-Delta-X}}\\
&  =\frac{1}{a} \int_{g}\{(Xf)\overline{g}+f\overline{Xg}\}\rho_{a}\,dx\\
&  =\frac{1}{a}\int_{G}\{(Zf)\overline{g}+f\overline{Zg}\}\rho_{a}\,dx
  && \text{see (\ref{Z-X-poly})}\\
  &  =\frac{1}{a} (Zf, g)_{L^2} + (f,Zg)_{L^2} \\
  &= \frac{2}{a} (Zf, g)_{L^2} && \text{by Corollary \ref{Z-symmetric}.}
\end{align*}
Since $Bf, Zf$ are both holomorphic and $\mathcal{P}$ is dense in
$\mathcal{H} L^2(\rho_a)$, we conclude $Bf = \frac{2}{a} Zf$.  

Now let $f \in \mathcal{D}(B)$ be arbitrary.  Since $\mathcal{P}$ is a
core for $B$, we may find $f_n \in \mathcal{P}$ with $f_n \to f$ and
$B f_n \to Bf$ in $L^2$, and also uniformly on compact sets.  In
particular, $Z f_n$ converges uniformly on compact sets so its limit
must be $Zf$.  We conclude that $Bf = \frac{2}{a} Zf$, and have also
shown the $\subset$ inclusion of (\ref{ZB-domain})

For the other inclusion, suppose $f, Zf \in \mathcal{H} L^2$, and as
in (\ref{e.d2.17}), set
\begin{equation*}
  g_n(z) = \int_{-\pi}^{\pi} F_n(\theta) f(\delta_{e^{i \theta}}(z))\,d\theta.
\end{equation*}
We showed in Theorem \ref{poly-density}(\ref{P-dense}) that $g_n \in
\mathcal{P}$ and $g_n \to
f$ in $L^2$.  
Since the integral is over a compact set and $f$ is smooth, we can
differentiate under the integral sign to obtain
\begin{equation*}
  Z g_n(z) = \int_{-\pi}^{\pi} F_n(\theta) (Zf)(\delta_{e^{i\theta}}(z))\,d\theta.
\end{equation*}
Then as before, we
have $Z g_n \to Zf$ in $L^2$.  Hence $B g_n \to \frac{2}{a} Zf$ in
$L^2$.  Since $B$ is a closed operator, we have $f \in \mathcal{D}(B)$.
\end{proof}

\begin{corollary}
\label{e-tB-dilate}
\begin{equation}
e^{-tB}f = f \circ \delta_{e^{-2t/a}} \label{e-tB-eqn}%
\end{equation}
for $f\in\mathcal{H}\cap L^{2}(\rho_{a})$ and $t \ge 0$.
\end{corollary}

\begin{proof}
For $f \in \mathcal{P}_k \subset \mathcal{D}(B)$, by Theorem
\ref{ZB} and (\ref{e.d2.13}), both sides of (\ref{e-tB-eqn}) are
equal to $e^{-2tk/a} f$.  Hence (\ref{e-tB-eqn}) holds for all $f \in
\mathcal{P}$.  Now if $f \in \mathcal{H} \cap L^2(\rho_a)$, by Theorem
\ref{poly-density}(\ref{P-dense}) we may choose $f_n \in \mathcal{P}$ with
$f_n \to f$ in $L^2(\rho_a)$.  Since $e^{-tB}$ is a contraction on
$L^2$, we have $e^{-tB} f_n \to e^{-tB} f$ in $L^2$, and also $f_n
\circ \delta_{e^{-2t/a}} \to f \circ \delta_{e^{-2t/a}}$ pointwise.
\end{proof}

\begin{remark}
  In light of Theorem \ref{ZB}, our goal of understanding strong
  hypercontractivity for the holomorphic projection of the semigroup
  $e^{-{\nate t}A}$ has essentially reduced to the problem of understanding it
  for the dilation semigroup on $G$.  A related study was undertaken
  in the papers \cite{gkl2010, gkl2015}, in which the authors consider
  the dilation semigroup on real Euclidean space.  In these papers,
  the holomorphic functions are replaced with the class of
  log-subharmonic functions, and the authors examine the relationship
  between an appropriate version of strong hypercontractivity and a
  so-called strong logarithmic Sobolev inequality for such functions.
  In recent work by the first author \cite{eldredge-slsi-lsh}, these
  results are extended to real stratified Lie groups.
\end{remark}

{\nate
\begin{remark}
  The dilation semigroup also arises from the Ornstein--Uhlenbeck
  semigroup $e^{-tA}$ in another way.  In
  \cite{lust-piquard-ornstein-uhlenbeck}, the author introduces a
  ``Mehler semigroup'' $e^{-tN}$ on a stratified Lie group, defined as
  follows (after adjusting notation and time scaling):
  \begin{equation}\label{mehler}
    (e^{-tN} f)(x) = \int_G f\left(\delta_{e^{-\beta t}} (x) \cdot
    \delta_{\sqrt{1-e^{-2 \beta t}}} (y)\right)\, \rho_a(y)\,m(dy)
  \end{equation}
  where we take $\beta = 2/a$ to make our time scaling come out right.
  The name ``Mehler semigroup'' is explained by the fact that when $G
  = \mathbb{R}^n$ (i.e. a stratified Lie group of step 1), then
  (\ref{mehler}) is precisely Mehler's formula for the
  Ornstein--Uhlenbeck semigroup, so in this special case, $e^{-tN} =
  e^{-tA}$.  For a non-abelian group $G$, $e^{-tN}$ and $e^{-tA}$
  differ, and $e^{-tN}$ is a non-symmetric semigroup on $L^2(\rho_a)$.
  A simple computation shows that, formally, the generator of
  $e^{-tN}$ is $N = -\Delta + \beta X = -\Delta + \frac{2}{a} X$.  In
  particular, when $f$ is holomorphic, we have (still formally)
  \begin{equation}
    Nf = \frac{2}{a} Xf = \frac{2}{a} Zf = Bf.
  \end{equation}
  Thus our main Theorem \ref{strong-hypercontractive} below could be
  restated as giving the strong hypercontractivity of the Mehler
  semigroup $e^{-tN}$, still conditionally on the logarithmic Sobolev
  inequality (\ref{LSI}).
\end{remark}
}

As a consequence of Theorem \ref{ZB}, we can show:

\begin{theorem}\label{A-not-holomorphic}
  Except in the abelian case $G = \mathbb{C}^n$, $A$ is not holomorphic.
\end{theorem}

\begin{proof}

  Consider the decomposition $\mathfrak{g} = \bigoplus_{j=1}^m V_j$ as
  in (\ref{e.d2.1}), where $V_m \ne 0$ is the center of
  $\mathfrak{g}$.  Excluding the abelian case $G = \mathbb{C}^n$, we
  have $m > 1$.  

  Fix a nonzero $\eta \in V_m$ and let $\ell : \mathfrak{g} \to
  \mathbb{C}$ be a complex linear functional with $\ell(\eta) = 1$ and
  $\ell = 0$ on $V_1 \oplus \dots \oplus V_{m-1}$.  The exponential
  map $\exp : \mathfrak{g} \to G$ is a holomorphic diffeomorphism, so
  we can define a holomorphic function $f : G \to \mathbb{C}$ by
  $f(\exp(\xi)) = \ell(\xi)$.  (Previously we took $G = \mathfrak{g}$
  as sets, and $\exp$ to be the identity, but for now we shall write
  $\exp$ explicitly.)  In fact, $f$ is homogeneous of degree $m$, so
  $f \in \mathcal{P}_m$.  We thus have $f \in \mathcal{D}(A) \cap
  \mathcal{D}(B)$ by Theorem \ref{poly-density}(\ref{P-core-B}) and
  Proposition \ref{P-DA}.  If $Af$ were holomorphic, by Proposition
  \ref{B-holomorphic-projection} we would have $Af = Bf$.  We show
  this is not the case.

  Let $g = \exp(\eta) \in G$, so that $f(g) = 1$.  By Theorem
  \ref{ZB} and (\ref{e.d2.13}), we have $Bf = \frac{2}{a} Zf =
  \frac{2m}{a} f$, so $Bf(g) = \frac{2m}{a}$.

  On the other hand, suppose $\xi \in V_1$.  For any $t \in
  \mathbb{R}$, we have $g \cdot \exp(t \xi) = \exp(\eta) \exp(t \xi) =
  \exp(\eta + t \xi)$, since $\eta \in V_m$ commutes with $\xi$.
  Thus $f(g \cdot \exp(t \xi)) = \ell(\eta + t \xi) = 1$ since
  $\xi \in V_1$ implies $\ell(\xi) = 0$.  Differentiating with respect
  to $t$ at $t=0$, we have $\widetilde{\xi} f(g) = 0$.  Hence
  $\nabla f(g) = 0$ and so by (\ref{A-smooth})    and (\ref{sub-Laplacian-def}),
   $Af(g) = 0 \ne Bf(g)$.
\end{proof}

As an explicit example, in the complex Heisenberg group
$\mathbb{H}_3^{\mathbb{C}}$ with coordinates $(z_1, z_2, z_3)$, one
could take $f(z) = z_3$ and verify by direct computation that
$Zf(0,0,1) = 2$ while $Af(0,0,1) = 0$.

In the case of stratified Lie groups of step 2, explicit integral
formulas for the heat kernel $\rho_a$ are known
\cite{gaveau77,taylor}.  So in those cases, to show $A$ is not
holomorphic, in light of (\ref{A-smooth}) one could compute $\bar{Z}_j
\log \rho_a$ and check that it is not holomorphic.

\section{Contractivity of $e^{-tB}$}\label{sec:contractivity}

\begin{theorem}\label{contractive}
  Let $0 < p < \infty$.  For every $f \in \mathcal{H}L^p(\rho_a)$ and
  every $t \ge 0$ we have
  \begin{equation}\label{contractive-eq}
    \|f \circ \delta_{e^{-t}}\|_{L^p(\rho_a)} \le \|f\|_{L^p(\rho_a)}.
  \end{equation}
  In particular, $e^{-tB}$ extends continuously to $\mathcal{H}
  L^p(\rho_a)$ for $0 < p < 2$, and is a contraction
  on $\mathcal{H} L^p(\rho_a)$ for $0 < p < \infty$.
\end{theorem}

\begin{proof}
  First, let us note that for any $g \in L^1(\rho_a)$, 
 the scaling relation  (\ref{semigroup-dilate}) implies
  \begin{equation}
    \int_G (g \circ \delta_{e^{-t}}) \,\rho_a\,dm = \int_G g\, \rho_{a e^{-2t}}\,dm.
  \end{equation}
  So if $g \in C^2_p(G)$ with $\Delta g \ge 0$, then Lemma
  \ref{ds-Delta-X} implies that this quantity decreases with respect
  to $t$; that is,
  \begin{equation}\label{g-decrease}
    \int_G (g \circ \delta_{e^{-t}}) \,\rho_a\,dm \le \int_G g\,
    \rho_{a}\,dm, \quad g \in C^2_p(G), \quad \Delta g \ge 0.
  \end{equation}

  We would now like to replace $g$ with some approximation of
  $|f|^p$.  To achieve this, let us first suppose that $f \in
  \mathcal{P}$; the general case will then follow from a density
  argument.  Following \cite[Lemma
    4.3]{gross-strong-hypercontractivity} we shall introduce a
  sequence of ``subharmonizing'' functions.

  Let $v \in C^\infty_c((0,\infty))$ be nonnegative, and set
  \begin{equation*}
    u(t) = \int_0^t \frac{1}{s} \int_0^s v(\sigma)\,d\sigma\,ds
  \end{equation*}
  Then it is easy to verify that:
  \begin{itemize}
  \item $u \in C^\infty([0,\infty))$;
  \item $u \ge 0$;
  \item $u', u''$ are bounded;
  \item $t u''(t) + u'(t) = v(t) \ge 0$ for all $t \ge 0$.
  \end{itemize}
  As such, if $f \in \mathcal{P}$ then $g := u(|f|^2) \in C^2_p(G)$.
  Now using the chain rule and the fact that $f$ is holomorphic (so
  that $\bar{Z}_j f = 0$), we have
  \begin{align*}
    \frac{1}{4}\Delta g  &  =
    \sum_{j=1}^{m} Z_{j} \bar{Z}_{j}u(|f|^2) \\
    &  =
    \sum_{j=1}^{m} Z_{j} \left[u'(|f|^{2}) f \overline{Z_{j}%
        f}\right]\\
    &  =\sum_{j=1}^{m}\left\{u''(|f|^{2}) \bar{f} Z_j f \cdot f
    \overline{Z_j f} + u' (|f|^{2})|Z_{j}f|^{2}\right\}\\
         &  =\sum_{j=1}^{m}\left(|f|^2 u''(|f|^{2}) +u'(|f|^{2}%
         )\right)|Z_{j}f|^{2}.
  \end{align*}
  Since $t u''(t) + u'(t) \ge 0$, we have $\Delta g \ge 0$ and so
  (\ref{g-decrease}) holds with $g = u(|f|^2)$.
  
  Now let $v_n \in C^\infty_c((0,\infty))$ be a sequence of
  nonnegative smooth functions with $v_n(\sigma) \uparrow
  \left(\frac{p}{2}\right) ^2
  \sigma^{(p/2)-1}$ for $\sigma > 0$, and as before set $    u_n(t) =
  \int_0^t \frac{1}{s} \int_0^s v_n(\sigma)\,d\sigma\,ds$ and $g_n =
  u_n(|f|^2)$.  As before, $g_n$ satisfies (\ref{g-decrease}).  
  By monotone convergence,
  \begin{equation*}
    u_n(t) \uparrow \int_0^t
  \frac{1}{s} \int_0^s \left(\frac{p}{2}\right) ^2
  \sigma^{(p/2)-1} \,d\sigma\,ds = t^{p/2}.
  \end{equation*}
  and hence $g_n \uparrow |f|^p$.  Hence using (\ref{g-decrease}) and
  monotone convergence, we have
  \begin{equation}\label{fp-decrease}
    \int_G |f \circ \delta_{e^{-t}}|^p \,\rho_a\,dm \le \int_G |f|^p
    \rho_{a}\,dm
  \end{equation}
  so that (\ref{contractive-eq}) holds for $f \in \mathcal{P}$.
  
  Now let $f \in \mathcal{H} L^p(\rho_a)$ be arbitrary.  As mentioned
  following Notation \ref{HLp-def}, $\mathcal{P}$ is dense in
  $\mathcal{H} L^p(\rho_a)$, so we may find a sequence $f_n \in
  \mathcal{P}$ with $f_n \to f$ in $L^p$ and also pointwise, so
  that in  particular $f_n \circ \delta_{e^{-t}} \to f \circ
  \delta_{e^{-t}}$ pointwise.  Now since (\ref{contractive-eq}) holds
  for $f_n$, we see that $f_n \circ
  \delta_{e^{-t}}$ is Cauchy in $L^p$, hence converges in $L^p$, and
  the limit must equal the pointwise limit $f \circ \delta_{e^{-t}}$.
  (In particular, $f \circ \delta_{e^{-t}} \in \mathcal{H} L^p(\rho_a)$.)
  Since the $p$-norm is continuous on $L^p$,  we can pass to the limit
  in (\ref{contractive-eq}) to see that it holds for $f$.
\end{proof}

\begin{corollary}
  $e^{-tB}$ is a strongly continuous contraction semigroup on
  $\mathcal{H} L^p(\rho_a)$ for $0 < p < \infty$.
\end{corollary}

\begin{proof}
  As we noted, $e^{-tB} f = f \circ \delta_{e^{-t}}$.  Hence the
  semigroup property is given by (\ref{e.d2.5}), and the previous theorem
  showed the contractivity.  To verify strong continuity, we note that
  for $f \in \mathcal{P}_k$ we have $f \circ \delta_{e^{-t}} \to f$
  pointwise, and $|f \circ \delta_{e^{-t}}| = e^{-tk} |f| \le |f|$.
  So by dominated convergence, $e^{-tB} f = f \circ \delta_{e^{-t}}
  \to f$ in $L^p$ as $t \to 0$.  By linearity, the same holds for any
  $f \in \mathcal{P}$.  For general $f \in \mathcal{H} L^p(\rho_a)$,
  we use a familiar triangle inequality argument.  Since $\mathcal{P}$
  is dense in $\mathcal{H} L^p$, for any $\epsilon$ we can choose $g
  \in \mathcal{P}$ with $\|f -g\|_{L^p} < \epsilon$.  For $p \ge 1$,
  Minkowski's triangle inequality gives
  \begin{align*}
    \|e^{-tB} f - f\|_{L^p}  &\le \|e^{-tB} (f - g)\|_{L^p} +
    \|e^{-tB} g - g\|_{L^p} + \|g - f\|_{L^p} 
    \\ &\le 2 \epsilon + \|e^{-tB} g - g\|_{L^p}
  \end{align*}
  using the contractivity of $e^{-tB}$ on the first term.  Since $g
  \in \mathcal{P}$, we know that $\|e^{-tB} g - g\|_{L^p} \to 0$ and
  hence $\limsup_{t \to 0} \|e^{-tB} f - f\|_{L^p} \le 2 \epsilon$,
  implying the desired result since $\epsilon$ is arbitrary.  For $0 <
  p < 1$, $\|\cdot\|_{L^p}$ is not a norm, but we get the same result
  by replacing $\|\cdot\|_{L^p}$ with $\|\cdot\|_{L^p}^p$ which does
  satisfy the triangle inequality.
\end{proof}

\section{Strong hypercontractivity for the dilation semigroup}\label{sec:strong}
We now state and prove our main theorem.  

We say that the heat kernel $\rho_a$ satisfies a \defword{logarithmic Sobolev
  inequality} if there exist $c > 0$ and $\beta \ge 0$ such
that 
\begin{equation}\label{LSI}
\int_G |f|^2\log|f|\rho_a\,dm
      \le cQ(f) + \beta \|f\|^2_{L^2(\rho_a)}
	    + \|f\|^2_{L^2(\rho_a)} \log\|f\|_{L^2(\rho_a)}
\end{equation}
for all $f$ such that $Q(f) < \infty$.  (In the case $\beta > 0$,
(\ref{LSI}) is sometimes called a \defword{defective} logarithmic
Sobolev inequality.)

{\nate
\begin{remark}
  To the best of our knowledge, it is currently an open problem to
  determine whether the logarithmic Sobolev inequality (\ref{LSI}) is
  satisfied in all complex stratified Lie groups $G$.  As such, our
  main Theorem \ref{strong-hypercontractive} is necessarily
  conditional in nature, taking (\ref{LSI}) as a hypothesis.  However,
  in Section \ref{sec:heis} below, we discuss the particular case of
  the complex Heisenberg and Heisenberg--Weyl groups, for which
  (\ref{LSI}) is known to hold \cite{eldredge-gradient,
    li-gradient-h-type}, and which therefore serve as a concrete
  example to which our theorem applies.  It would be of great interest
  to have additional examples of groups satisfying (\ref{LSI}).
\end{remark}
}

For $0 < q \le p < \infty$, let
\begin{align}
t_J(p,q) &:= \frac{c}{2}\log\left(\frac{p}{q}\right)              \label{tJ-def} \\
\intertext{and}
M(p,q) &:= \exp\left(2\beta\left(\frac{1}{q}-\frac{1}{p}\right)\right)     \label{Mpq-def}  
\end{align}

\begin{theorem} \label{strong-hypercontractive}
  Suppose that the logarithmic Sobolev inequality (\ref{LSI}) holds and that
  $0 < q \le p < \infty$. Then for every $f \in \mathcal{H}
  L^q(\rho_a)$ and every $t \ge t_J(p,q)$,
  \begin{equation}
    \|e^{-tB}f\|_{L^p(\rho_a)} \le M(p,q)\|f\|_{L^q(\rho_a)}
    \label{strong-hypercontractive-eq}
  \end{equation}
\end{theorem}

\begin{proof}
  Fix $0 < q \le p < \infty$.  We shall concentrate first on the case
  when $f \in \mathcal{P}$; let us say $f$ has degree $D$, so $f \in
  \bigoplus_{k=0}^D \mathcal{P}_k$.  The general case will then follow
  by a density argument as in the proof of Theorem \ref{contractive}.
  We also note that it is sufficient to
  prove that  (\ref{strong-hypercontractive-eq}) holds
  for $t = t_J(p,q)$, since if this can be shown then using Theorem
  \ref{contractive} we conclude that for any $t \ge t_J$,
  \begin{equation*}
    \|e^{-tB}f\|_{L^p} = \|e^{-t_J B} (e^{-(t-t_J) B} f)\|_{L^p} \le
    M(p,q) \|e^{-(t-t_J)B} f\|_{L^q} \le M(p,q) \|f\|_{L^q}.
  \end{equation*}

  We adopt similar notation as in
  \cite[Section 4]{gross-hypercontractivity-complex}, which we
  generally follow.  Let
  \begin{equation*}
    g_t := e^{-tB} f.
  \end{equation*}
  Since $\mathcal{P}_k$ is invariant under $B$ (Corollary \ref{e-tB-dilate} and
  Lemma \ref{l.d2.11}), $g_t$ is a smooth curve in the
  finite-dimensional space $\bigoplus_{k=0}^D \mathcal{P}_k$.
  Indeed, if $f = \sum_{k=0}^D f_k$ with  $f_k \in \mathcal{P}_k$, we
  have $g_t = \sum_{k=0}^D e^{-2tk/a} f_k$.

  Fix $\epsilon > 0$ and let
  \begin{align*}
    \gamma_t &:= \left(|g_t|^2 + \epsilon\right)^{1/2} \\
    r(t) &:= qe^{2t/c} \\
    v(t) &:= \int \gamma_t(x)^{r(t)} \rho_a(x)\,m(dx) \\
    \alpha(t) &:= \|\gamma_t\|_{L^{r(t)}(\rho_a)} = v(t)^{1/r(t)}.
  \end{align*}
  Notice that $\gamma_t \in C^2_p(G)$ (see Notation \ref{poly-growth})
  and in particular $v(t)$, $\alpha(t)$ are finite for all $t$.  Also
  notice that $r(t_J) = p$.  Our goal will be to show $\alpha(t_J) \le
  M(p,q) \alpha(0)$, which when taking $\epsilon \to 0$ turns into (\ref{strong-hypercontractive-eq})
  with $t = t_J$.  We will do this by deriving an
  appropriate differential inequality for $\alpha$.

  Simple calculus shows
  \begin{equation}\label{alpha-prime}
    \alpha'(t) = \alpha(t) v(t)^{-1} 
           \left(r(t)^{-1} v'(t) - \frac{2}{c} v(t) \log \alpha(t)\right).
  \end{equation}
  To attack this, we
  differentiate under the integral sign to show
  \begin{align}
    v'(t) &= \int_G \gamma_t^{r(t)} \left( r'(t) \log \gamma_t +
    \frac{r(t)}{\gamma_t} \gamma'_t\right)
    \rho_a\,dm \label{vprime-first} \\ 
    &= \frac{2r(t)}{c}  \int_G \gamma_t^{r(t)} \log \gamma_t\, \rho_a\,dm +
    r(t) \int_G \gamma_t^{r(t)-1} \gamma'_t\, \rho_a\,dm  &&
     \\
    &= \frac{2r(t)}{c} \int_G \gamma_t^{r(t)} \log \gamma_t \,\rho_a\,dm -
    r(t) \Re \int_G \gamma_t^{r(t)-2} Bg_t \cdot \overline{g_t}\,
    \rho_a\,dm. \label{vprime-last}
  \end{align}
  To check that differentiation under the integral sign is justified,
  fix a bounded interval $[t_1, t_2]$ containing $t$, and note that
  since $s \mapsto g_s$ is a continuous curve in the holomorphic
  polynomials of degree $D$, there is a constant $C$ so that $|g_s(x)|
  + |g'_s(x)| \le C (1+d(e,x))^D$ for all $s \in [t_1, t_2]$.  Since $\gamma_t$ is bounded below
  and $r, r'$ are bounded on $[t_1, t_2]$ by some constant $R$, it
  follows that for $t \in [t_1, t_2]$ the integrand on the right side
  of (\ref{vprime-first}) is dominated by some constant times
  $(C(1+d(e,x))^D)^{R+1} \rho_a(x)$,  which is integrable.

  Let $I := r \Re \int_G \gamma^{r-2} Bg \cdot \bar{g}\,
  \rho_a\,dx$ be the second term in (\ref{vprime-last}).  (For
  notational hygiene, we suppressed the explicit dependence on $t$,
  and will continue to do so when convenient.)  We wish to estimate
  $I$ from below using the logarithmic Sobolev inequality, so we need
  to convert it into an expression involving $Q$.

  Since $g$ is a polynomial, by Theorem \ref{ZB} and
  (\ref{Z-X-poly}), we have $B g = \frac{2}{a} Zg = \frac{2}{a} X g$, so that
  \begin{equation*}
    I = \frac{2r}{a} \Re \int_G \gamma^{r-2} Xg \cdot \bar{g}\, \rho_a\,dm.
  \end{equation*}
  But $X$ is a real vector field, so an easy computation shows $X
  [|g|^2] = 2 \Re [Xg \cdot \bar{g}]$ and hence $X [\gamma^{r}] =
  r \gamma^{r-2}  \Re [X g \cdot \bar{g}]$.  Since $\gamma^r \in
  C^2_p(G)$, by Lemma \ref{ds-Delta-X} we have
  \begin{equation*}
    I = \frac{2}{a} \int_G X[\gamma^{r}] \rho_a\,dm = 
    \int_G \Delta[\gamma^{r}] \rho_a\,dm.
  \end{equation*}
  Now using elementary calculus, we may show:
  \begin{equation}\label{calculus-claim}
    \Delta[\gamma^r] = 4 |\nabla \gamma^{r/2}|^2  +
    r \epsilon \gamma^{r-4} |\nabla g|^2.
  \end{equation}
  To see this, let $Z_j$ be the vector fields defined in (\ref{Zj-def}),
  which are of type $(1,0)$, so that $\Delta = 4\sum_j  Z_j \bar{Z}_j$.
  We have
  \begin{align*}
      4 Z_j \bar{Z}_j [\gamma^r] &= 4 Z_j \left[ \frac{r}{2} \gamma^{r-2}
        \cdot \left( \cancel{\bar{Z}_j g} \cdot \bar{g} + g \cdot \bar{Z}_j
        \bar{g} \right) \right] \\
      &= 2r \cdot \frac{r-2}{2} \gamma^{r-4} \cdot \left( Z_j g \cdot
        \bar{g} + g \cdot \cancel{Z_j \bar{g}}\right) (g \cdot \bar{Z}_j)
        \\
        & \quad + 2r \gamma^{r-2} \left(Z_j g \cdot \bar{Z}_j \bar{g} + g
        \cdot \cancel{Z_j \bar{Z}_j \bar{g}}\right)
        \end{align*}
since $Z_j \bar{Z}_j \bar{g} = \bar{Z}_j Z_j \bar{g} =
0$.  Now rearranging,
\begin{align*}
  4 Z_j \bar{Z}_j[ \gamma^r] &= r(r-2) \gamma^{r-4} |Z_j g |^2 |g|^2 + 2r \gamma^{r-2} |Z_j
        g|^2
        \\
        &= r^2 \gamma^{r-4} |Z_j g|^2 |g|^2 + 2r \gamma^{r-4} |Z_j
        g|^2 (\gamma^2 - |g|^2) \\
        &= r^2 \gamma^{r-4} |Z_j g|^2 |g|^2 + 2r \epsilon \gamma^{r-4} |Z_j
        g|^2
    \end{align*}
since $\gamma^2 - |g|^2 = \epsilon$.
    On the other hand,
    \begin{equation*}
      Z_j [\gamma^{r/2}] = \frac{r}{4} \gamma^{\frac{r-4}{2}} Z_j g \cdot \bar{g}
    \end{equation*}
    so that
    \begin{equation*}
      4 Z_j \bar{Z}_j [\gamma^r] = 16 |Z_j [\gamma^{r/2}]|^2 + 2r \epsilon
      \gamma^{r-4} |Z_j g|^2.
    \end{equation*}
    Summing over $j$ and referring to
    (\ref{gradsq-real}--\ref{gradsq-holo}), we obtain
    (\ref{calculus-claim}).

    In particular, since the second term of (\ref{calculus-claim}) is
    nonnegative,
    \begin{equation*}
      \Delta [\gamma^r] \ge 4 | \nabla[ \gamma^{r/2}]|^2.
    \end{equation*}
    So integrating gives
    \begin{equation*}
      I \ge 4 Q(\gamma^{r/2}).
    \end{equation*}

  Now applying the logarithmic Sobolev inequality (\ref{LSI}) and
  noting that $\big\lVert \gamma_t^{r(t)/2}\big\rVert_{L^2(\rho_a)}^2 = v(t)$, it follows
  that
  \begin{align*}
    I \ge \frac{2r(t)}{c} \int_G \gamma_t^{r(t)} \log \gamma_t\, \rho_a\,dm -
    \frac{4\beta}{c} v(t) - \frac{2}{c} v(t) \log v(t)
  \end{align*}
  Referring back to (\ref{vprime-last}), this shows
  \begin{equation}
    v'(t) \le \frac{4\beta}{c} v(t) + \frac{2}{c} v(t) \log v(t) =
    \frac{4\beta}{c} v(t) + \frac{2r(t)}{c} v(t) \log \alpha(t)
  \end{equation}
  and thus from (\ref{alpha-prime})
  \begin{equation}
    \alpha'(t) \le \frac{4 \beta \alpha(t)}{c r(t)}.
  \end{equation}
  In other words,
  \begin{equation}
    \frac{d}{dt} \log \alpha(t) \le \frac{4\beta}{c r(t)} = \frac{4\beta
    }{cq} e^{-2t/c} 
  \end{equation}
  so, integrating,
  \begin{equation}\label{alpha-ineq}
    \alpha(t) \le \alpha(0)
    \exp\left(\frac{2\beta}{q}(1-e^{-2t/c})\right) = \alpha(0) \exp\left(2
    \beta \left(\frac{1}{q} - \frac{1}{r(t)}\right)\right).
  \end{equation}
  Now let $\epsilon \downarrow 0$, so that $\gamma_t \downarrow
  |g_t|$, and by dominated convergence, 
  $\alpha(t) \downarrow \|g_t\|_{L^{r(t)}(\rho_a)} = \|e^{-tB} f \|_{L^{r(t)}(\rho_a)}$.  
  Taking $t = t_J$ and recalling that $r(t_J) = p$, (\ref{alpha-ineq})
  becomes
  \begin{equation}
    \|e^{-t_J B} f\|_{L^p(\rho_a)} \le M(p,q) \|f\|_{L^q(\rho_a)}
  \end{equation}
  which is precisely (\ref{strong-hypercontractive-eq}) with $t = t_J$.  This completes the
  proof for $f \in \mathcal{P}$.

  For general $f \in \mathcal{H} L^q(\rho_a)$, proceed as in the last
  paragraph of the proof of Theorem \ref{contractive}.  Choose a
  sequence $f_n \in \mathcal{P}$ with $f_n \to f$ in $L^q$-norm.  Then
  (\ref{strong-hypercontractive-eq}) holds for $f_n$.  As
  $n \to \infty$, the
  right side of (\ref{strong-hypercontractive-eq}) converges to $M(p,q) \|f\|_{L^q(\rho_a)}$.
  Since $e^{-tB}$ is a contraction on $\mathcal{H} L^p$ by Theorem
  \ref{contractive}, $e^{-tB} f_n$ is Cauchy in $L^p$ norm, so
  converges in $L^p$ to some function which can only be $e^{-tB} f$.
  Hence the left side of (\ref{strong-hypercontractive-eq}) converges to $\|e^{-tB}
  f\|_{L^p(\rho_a)}$ as desired.
\end{proof}

\section{Application to the complex Heisenberg group}\label{sec:heis}

In order for Theorem \ref{strong-hypercontractive} to have content, we
need examples of stratified complex groups for which the logarithmic
Sobolev inequality (\ref{LSI}) is satisfied.  In this section, we
verify that the complex Heisenberg group $\mathbb{H}_3^{\mathbb{C}}$
of Examples \ref{ex-heis} and \ref{ex-heis-metric} enjoys that
property, as do the complex Heisenberg--Weyl groups
$\mathbb{H}_{2n+1}^{\mathbb{C}}$ of Examples \ref{ex-heis-weyl} and
\ref{ex-heis-weyl-metric}.  So for these groups, the hypotheses of our
Theorem \ref{strong-hypercontractive} are satisfied.  On the other
hand, since as shown in Theorem \ref{A-not-holomorphic}, the
operator $A$ is not holomorphic in this setting, the results of
\cite{gross-hypercontractivity-complex} do not apply, so we have
proved something new.

{\nate Indeed, the papers \cite{eldredge-gradient}  \nate and
\cite{li-gradient-h-type}  showed independently that so-called H-type
Lie groups satisfy a gradient estimate which is known to imply the
logarithmic Sobolev inequality (\ref{LSI}).}  We shall state that
result, check that the complex Heisenberg group
$\mathbb{H}_3^{\mathbb{C}}$ is an H-type Lie group, and sketch in the
steps leading to (\ref{LSI}).  The same argument, \emph{mutatis
  mutandis}, also applies to the Heisenberg-Weyl groups
$\mathbb{H}_{2n+1}^{\mathbb{C}}$; we omit the details because they add
notation but no further insight.

\begin{definition}\label{h-type-def}
  Suppose $\mathfrak{g}$ is a \emph{real} Lie algebra equipped with a
  positive definite inner product $\langle \cdot, \cdot \rangle$.  For
  $u,v \in \mathfrak{g}$, define $J_u v$ via
\begin{equation*}
  \langle J_u v, w \rangle = \langle u, [v,w] \rangle.
\end{equation*}
Let $\mathfrak{z}$ be the center of $\mathfrak{g}$, and $\mathfrak{v}
= \mathfrak{z}^\perp$.  We say $(\mathfrak{g}, \langle \cdot, \cdot
\rangle)$ is \defword{H-type} if:
\begin{enumerate}
\item \label{h-type-decomp} $[\mathfrak{v}, \mathfrak{v}] = \mathfrak{z}$; and
\item \label{h-type-rigid} For each $u \in \mathfrak{z}$ with $\|u\|=1$, $J_u$ maps
  $\mathfrak{v}$ isometrically onto itself.
\end{enumerate}

An \defword{H-type Lie group} is a connected, simply connected real Lie
group $G$ equipped with an inner product $\langle \cdot, \cdot
\rangle$ on its Lie algebra $\mathfrak{g}$ such that $(\mathfrak{g},
\langle \cdot , \cdot \rangle)$ is H-type in the above sense.
\end{definition}

Suppose then that $(G, \langle \cdot, \cdot \rangle)$ is an H-type Lie
group.  By item \ref{h-type-decomp} of definition \ref{h-type-def},
$G$ is nilpotent, so we may fix a bi-invariant Haar measure $m$ which
is simply (a scalar multiple of) Lebesgue measure.  Let $\xi_1, \dots,
\xi_n$ be an orthonormal basis for $\mathfrak{v} \subset
\mathfrak{g}$, let $\widetilde{\xi_1}, \dots, \widetilde{\xi_n}$ be
the corresponding left-invariant vector fields, and define the
sub-Laplacian by $\Delta = \widetilde{\xi_1}^2 + \dots +
\widetilde{\xi_n}^2$.  Also, for sufficiently smooth $f$ let $|\nabla
f|^2 := |\widetilde{\xi_1} f|^2 + \dots + |\widetilde{\xi_n} f|^2$.
The main theorem of \cite{eldredge-gradient} {\nate and
\cite{li-gradient-h-type}} is:

\begin{theorem}\label{h-type-gradient}
  If $(G, \langle \cdot, \cdot \rangle)$ is H-type, then following the
  above notation, there is a constant $K$ such that for all $t \ge 0$
  and $f \in C^1_c(G)$
  we have
  \begin{equation}
    |\nabla e^{t \Delta/4} f| \le K e^{t \Delta/4} |\nabla f|.
  \end{equation}
\end{theorem}

\begin{lemma}
  Consider $\mathbb{H}_3^{\mathbb{C}}$ as a 6-dimensional real Lie
  group.  As a set, $\mathfrak{h}_3^{\mathbb{C}} = \mathbb{C}^3 =
  \mathbb{R}^6$, so equip it with the Euclidean inner product $\langle
  \cdot, \cdot \rangle$.  Then $(\mathbb{H}_3^{\mathbb{C}}, \langle
  \cdot, \cdot \rangle)$ is H-type.
\end{lemma}

\begin{proof}
  Let $\{e_j, ie_j : j = 1,2,3\}$ be the standard basis of
  $\mathfrak{h}_3^{\mathbb{C}} = \mathbb{C}^3 = \mathbb{R}^6$, which
  is orthonormal with respect to the (real) Euclidean inner product
  $\langle \cdot, \cdot \rangle$.  Then the center $\mathfrak{z}$ of
  $\mathfrak{h}_3^{\mathbb{C}}$ is spanned (over $\mathbb{R}$) by $\{e_3, ie_3\}$, so
  $\mathfrak{v} = \mathfrak{z}^\perp$ is spanned by $\{e_1, ie_1, e_2,
  ie_2\}$.  By inspection of the Lie bracket defined in
  (\ref{heis-bracket}), we see that $[\mathfrak{v}, \mathfrak{v}] =
  \mathfrak{z}$.

  Next, we note that for $u,v,w \in \mathfrak{h}_3^{\mathbb{C}}$ and
  $\alpha, \beta \in \mathbb{C}$, we have
  \begin{equation}
    \langle J_{\alpha u} (\beta v), w \rangle = \langle \alpha u, [ \beta v, w] \rangle = \langle u, 
            [ v, \bar{\alpha} \beta w] \rangle = \langle J_u v,
            \bar{\alpha} \beta w \rangle = \langle \alpha \bar{\beta}
            J_u v, w \rangle
  \end{equation}
  so that $J_u v$ is complex-linear in $u$ and conjugate-linear in
  $v$.  Together with the relations $J_{e_3} e_1 = e_2$, $J_{e_3} e_2
  = -e_1$, we easily see that for any $\alpha \in \mathbb{C}$ with
  $|\alpha| = 1$, we have that $J_{\alpha e_3}$ is an isometry of
  $\mathfrak{v}$ into itself.  
\end{proof}

Now we note that when the dual metric $h$ is defined on
$(\mathfrak{h}_3^{\mathbb{C}})^*$ as in Example \ref{ex-heis-metric},
the backward annihilator $H$ is precisely $\mathfrak{v}$, and the
metric $g$ is just the restriction of $\langle \cdot, \cdot \rangle$
to $H$.  Hence the sub-Laplacian $\Delta$ used in Theorem
\ref{h-type-gradient} is the same as that defined in
(\ref{sub-Laplacian-def}), and for smooth real $f$, the squared
gradient $|\nabla f|$ of Theorem \ref{h-type-gradient} is equal to
$h(df, df)$ in the notation of Section \ref{sub-riemannian-section}.

\begin{theorem}
  It follows from Theorem \ref{h-type-gradient} that the logarithmic
  Sobolev  inequality (\ref{LSI}) holds for $\mathbb{H}_3^{\mathbb{C}}$, with $c =
  2 K^2 a$ and $\beta = 0$, where $K$ is the constant from Theorem \ref{h-type-gradient}.
\end{theorem}

\begin{proof}
   This can be proved by an elementary, though clever, argument in the
   style of $\Gamma_2$-calculus, which can be found in \cite[Theorem
     6.1]{bbbc-jfa}.  The essence of this argument, which is an
   equivalence between gradient bounds and the logarithmic Sobolev
   inequality, goes back to \cite{bakry-emery-short}.
\end{proof}

\begin{corollary}
  Theorem \ref{strong-hypercontractive} holds for the complex Heisenberg and
  Heisenberg--Weyl groups $\mathbb{H}_{2n+1}^{\mathbb{C}}$, with
  $t_J(p,q) = K^2 a \log\left(\frac{p}{q}\right)$ and $M(p,q) = 1$,
  where $K$ is the constant from Theorem \ref{h-type-gradient}.
\end{corollary}

\begin{remark}
  The foregoing argument would apply to any complex stratified Lie
  group which is H-type.  Since the complex stratified groups and the
  H-type groups are each rather large classes, one might think there
  would be many more such examples.  However, there are actually no
  more: the first author has shown in \cite{eldredge-complex-h-type}
  that the complex Heisenberg--Weyl Lie algebras are the only complex
  Lie algebras which are H-type under a Hermitian inner product.
\end{remark}

\noindent\textit{Acknowledgments.}  The authors would like to thank
Bruce K.~Driver for helpful discussions regarding this paper.  We
would also like to thank the anonymous referees for several very
helpful suggestions, including the relevance of the paper
\cite{lust-piquard-ornstein-uhlenbeck}.

The first author's research was supported by a grant from the Simons
Foundation (\#355659, Nathaniel Eldredge).  The third author's
research was supported in part by National Science Foundation grant
DMS 1404435.

\bibliographystyle{plainnat}
\bibliography{allpapers}

\end{document}